\documentclass[a4paper,11pt]{article}

\usepackage[utf8]{inputenc}
\usepackage[T1]{fontenc}
\usepackage{graphicx}
\usepackage{amsmath,amsfonts,amssymb,amsthm}
\usepackage{mathtools}
\usepackage{mathrsfs}
\usepackage{xcolor}
\usepackage{hyperref}
\usepackage{booktabs}
\usepackage{bbold}
\usepackage{bm}
\usepackage{color}
\usepackage[font=footnotesize]{caption}
\usepackage{enumitem}
\usepackage{empheq}
\usepackage{framed}
\usepackage{fullpage}
\usepackage{hyperref}
\usepackage{soul}
\usepackage{dirtytalk}

\mathtoolsset{showonlyrefs}

\definecolor{myred}{HTML}{880000}
\definecolor{mygreen}{HTML}{008800}
\definecolor{myblue}{HTML}{000088}
\definecolor{linkblue}{HTML}{0000BB}

\hypersetup{
  colorlinks,
  linkcolor={myred},
  citecolor={myblue},
  urlcolor={linkblue}
}

\newcommand{\E}{{\mathbf E}}

\newcommand{\Var}{\operatorname{Var}}

\renewcommand{\le}{\leqslant}
\renewcommand{\ge}{\geqslant}

\renewcommand{\ln}{\log}

\DeclareMathOperator{\tr}{Tr}

\newcommand{\var}{\Var}

\newcommand{\ind}[1]{\bm 1 ( #1 )}

\renewcommand{\top}{\mathsf{T}}

\usepackage{xspace}

\newtheorem{proposition}{Proposition}
\newtheorem{theorem}{Theorem}
\newtheorem*{theorem*}{Theorem}
\newtheorem{lemma}{Lemma}

\theoremstyle{definition}
\newtheorem{definition}{Definition}

\newtheorem{example}{Example}

\theoremstyle{remark}
\newtheorem{remark}{Remark}

\title{Dimension-free Bounds for Sums of Independent Matrices and Simple Tensors via the Variational Principle}
\author{
  Nikita Zhivotovskiy\thanks{Department of Mathematics, ETH Z\"{u}rich, Switzerland, \href{mailto:nikita.zhivotovskii@math.ethz.ch}{nikita.zhivotovskii@math.ethz.ch}}
}

\begin{document}

\maketitle

\begin{abstract}
We consider the deviation inequalities for the sums of independent $d$ by $d$ random matrices, as well as rank one random tensors. Our focus is on the non-isotropic case and the bounds that do not depend explicitly on the dimension $d$, but rather on the effective rank.
In an elementary and unified manner, we show the following results: 
\begin{itemize}
\item A deviation bound for the sums of independent positive-semi-definite matrices. This result complements the dimension-free bound of Koltchinskii and Lounici [Bernoulli, 23(1): 110-133, 2017] on the sample covariance matrix in the sub-Gaussian case.
\item A new bound for truncated covariance matrices that is used to prove a dimension-free version of the bound of Adamczak, Litvak, Pajor and Tomczak-Jaegermann [Journal Of Amer. Math. Soc., 23(2), 535–561, 2010] on the sample covariance matrix in the log-concave case.
\item Dimension-free bounds for the operator norm of the sums of random tensors of rank one formed either by sub-Gaussian or by log-concave random vectors. This complements the result of Gu\'{e}don and Rudelson [Adv. in Math., 208: 798-823, 2007].
\item A non-isotropic version of the result of Alesker [Geom. Asp. of Funct. Anal., 77: 1--4, 1995] on the deviation of the norm of sub-exponential random vectors.
\item A dimension-free lower tail bound for sums of positive semi-definite matrices with heavy-tailed entries, sharpening the bound of Oliveira [Prob. Th. and Rel. Fields, 166: 1175–1194, 2016].
\end{itemize}
Our approach is based on the duality formula between entropy and moment generating functions.
In contrast to the known proofs of dimension-free bounds, we avoid Talagrand’s majorizing measure theorem, as well as generic chaining bounds for empirical processes. Some of our tools were pioneered by O. Catoni and co-authors in the context of robust statistical estimation. 
\end{abstract}

\section{Introduction and main results}

We study the non-asymptotic bounds for the sums of some independent random matrices as well as a closely related question of estimating the largest and smallest singular values of random matrices with independent rows. Assume that we are given a random $n$ by $d$ matrix $A$ such that all of its rows $A^{\top}_1, \ldots, A^{\top}_n$ are isotropic independent sub-Gaussian vectors (see the formal definitions below). We are interested in providing the upper and lower bounds on its singular values 
\[
s_{1}(A), s_2(A), \ldots, s_d(A).
\]
The question of upper bounding the largest singular values and lower bounding the smallest singular value is known to be essentially equivalent (see \cite[Chapter 4]{Vershynin2016HDP}) to providing an upper bound on the operator norm of the difference between the sample covariance matrix formed by the rows $A_i$ and the identity matrix. That is, one is interested in providing a high probability, non-asymptotic bound on
\begin{equation}
\label{eq:covestim}
\left\|\frac{1}{n}\sum_{i = 1}^nA_{i}A_{i}^{\top} - I_d\right\|.
\end{equation}
Here and in what follows $\left\|\cdot\right\|$ stands for the operator norm of the matrix and for the Euclidean norm of the vector respectively.
The latter question is also central in mathematical statistics, where one is interested in estimating the underlying covariance structure using the sample covariance matrix. One of the usual assumptions made when analyzing \eqref{eq:covestim} is that the rows $A^{\top}_i$ are isotropic and zero mean; that is, $\E A_i = 0$ and $\E A_{i}A_{i}^{\top} = I_d$, where in what follows $I_d$ stands for the $d$ by $d$ identity matrix. The non-isotropic case can usually be reduced to the isotropic using a linear transformation. However, the problem is that in this case the bound on \eqref{eq:covestim} will depend on the dimension $d$, whereas in many cases one expects that a dimension-free deviation bound is possible. The search for  dimension-free bounds for sums of independent random matrices is motivated mainly by applications in statistics and data science, where it is usually assumed that the data lives on a low-dimensional manifold. Before providing our first result, we recall that for a random variable $Y$ and $\alpha \in [1, 2]$, its $\psi_{\alpha}$ norm is defined as follows
\[
\|Y\|_{\psi_{\alpha}} = \inf\{c > 0: \E\exp(|Y|^\alpha/c^{\alpha}) \le 2\}.
\]
Using the standard convention, we say that $\|\cdot\|_{\psi_{2}}$ is the sub-Gaussian norm and $\|\cdot\|_{\psi_1}$ is the sub-exponential norm. We say that $X$ is a sub-Gaussian random vector in $\mathbb{R}^d$ if $\sup_{u \in S^{d - 1}}\|\langle u , X\rangle\|_{\psi_2}$ is finite. A zero mean random vector $X$ is isotropic, if $\E XX^{\top} = I_d$. Here and in what follows, $S^{d - 1}$ denotes the corresponding unit sphere and $\langle \cdot , \cdot\rangle$ is the standard inner product in $\mathbb{R}^d$. One of the central quantities appearing in this paper is the effective rank. 
\begin{definition}
For a positive semi-definite matrix $\Sigma$ define its \emph{effective rank} as
\[
\mathbf{r}(\Sigma) = \frac{\tr(\Sigma)}{\|\Sigma\|}.
\]
\end{definition}
The effective rank is always smaller than the matrix rank of $\Sigma$ and, in particular, smaller than its dimensions. We also have $\mathbf{r}(I_d) = d$.
Our first result is a general upper bound for sums of independent positive semi-definite $d$ by $d$ matrices satisfying the sub-exponential norm equivalence assumption. This generalizes the question of upper bounding \eqref{eq:covestim}, since we neither assume that the matrix is of rank one nor that the covariance matrix is identity.
\begin{theorem}[A general rank version of Theorem 9 in \cite{koltchinskii2017operators}]
\label{thm:maintheorem}
Assume that $M_1, \ldots M_n$ are independent copies of a $d$ by $d$ positive semi-definite symmetric random matrix $M$ with mean $\E M = \Sigma$. Let $M$ satisfy for some $\kappa \ge 1$,
\begin{equation}
\label{eq:psionellone}
\|x^{\top} M x\|_{\psi_1} \le \kappa^2\ x^{\top}\Sigma x,
\end{equation}
for all $x \in \mathbb{R}^d$.
Then, for any $t > 0$, with probability at least $1 - \exp(-t)$, it holds that
\begin{equation}
\label{eq:generalsamplecovariance}
\left\|\frac{1}{n}\sum\limits_{i = 1}^nM_i - \Sigma\right\| \le 20\kappa^2\|\Sigma\|\sqrt{\frac{4\mathbf{r}(\Sigma) + t}{n}},
\end{equation}
whenever $n \ge 4\mathbf{r}(\Sigma)+t$.
\end{theorem}

\begin{remark}
In the theorem above, we presented explicit constants. We place them to emphasize that, in contrast with existing dimension-free bounds, these constants are easy to obtain with the approach we follow. At the same time, little effort was made to get their optimal values. 
\end{remark}
\begin{remark}
In Section \ref{sec:lowertail} we show that the same dimension-free bound holds under a weaker assumption (allowing heavy-tailed distributions), namely $\sqrt{\E(x^{\top} M x)^2} \le \kappa^2\ x^{\top}\Sigma x$, but only for the lower tails of \eqref{eq:generalsamplecovariance}. This complements several known dimension-dependent lower tail bounds.
\end{remark}
The norm equivalence assumption \eqref{eq:psionellone} is quite standard in the literature. As a matter of fact, Theorem \ref{thm:maintheorem} recovers one of the central results in high-dimensional statistics, as the following example shows.
\begin{example}[The sample covariance matrix in the sub-Gaussian case \cite{koltchinskii2017operators}]
\label{ex:samplecovariance}
 The most natural application of Theorem \ref{thm:maintheorem} is when $M = XX^{\top}$ and $X$ is zero mean \emph{sub-Gaussian random vector} with a covariance matrix $\Sigma$. That is, there is $\kappa \ge 1$ such that for any $y \in \mathbb{R}^d$, it holds that
 \begin{equation}
 \label{eq:subgauss}
 \|\langle y, X\rangle\|_{\psi_2} \le \kappa\sqrt{y^{\top}\Sigma y}.
 \end{equation}
Using this line, for any $x \in \mathbb{R}^d$ we have
$\|x^{\top}XX^{\top}x\|_{\psi_1} = \|x^{\top}X\|_{\psi_2}^2 \le \kappa^2 x^{\top}\Sigma x,$
which is sufficient for Theorem \ref{thm:maintheorem}. This gives that, with probability at least $1 - \exp(-t)$, provided that $n \ge 4\mathbf{r}(\Sigma)+t$, it holds that
\begin{equation}
\label{eq:samplecovariance}
\left\|\frac{1}{n}\sum\limits_{i = 1}^nX_iX_i^{\top} - \Sigma\right\| = \sup\limits_{y \in S^{d - 1}}\left|\frac{1}{n}\sum\limits_{i = 1}^n \langle X_i, y\rangle^2 - \E \langle X, y\rangle^2 \right| \le 20\kappa^2\|\Sigma\|\sqrt{\frac{4\mathbf{r}(\Sigma) + t}{n}},
\end{equation}
where we used that for a symmetric $d$ by $d$ matrix $A$ it holds that $\|A\| = \sup\nolimits_{y \in S^{d - 1}} y^{\top}Ay$.
\end{example}
Recently, much attention has been paid to the dimension-free bound for the sample covariance matrix formed by sub-Gaussian random vectors. Although the dimension-dependent version of \eqref{eq:samplecovariance} (that is, where $\mathbf{r(\Sigma)}$ is replaced by $d$) follows from a simple discretization argument, the known approaches to obtaining the dimension-free bounds are quite technical and deserve a separate discussion:
\begin{itemize}
\item The bound of Theorem \ref{thm:maintheorem} is a general rank version\footnote{For the sake of technical simplicity, we work with finite dimensional vectors, whereas the results in \cite{koltchinskii2017operators} are formulated for a general Hilbert space.} of the result of Koltchinskii and Lounici \cite[Theorem 9]{koltchinskii2017operators}. Their proof is based on deep probabilistic results: the generic chaining tail bounds for quadratic processes and Talagrand's majorizing measures theorem. In particular, this makes it difficult to provide any explicit constants using their approach. 
Some generalizations of the result on sample covariance matrices to positive semi-definite matrices are also considered in \cite{youssef2013estimating}.
\item Using the matrix deviation inequality of Liaw, Mehrabian, Plan, and Vershynin \cite{liaw2017simple}, Vershynin \cite{Vershynin2016HDP} gives an alternative proof of the bound of Koltchinskii and Lounici, but with the term $\kappa^4$ instead of $\kappa^2$ in \eqref{eq:samplecovariance}. The dependence on $\kappa$ in \cite{liaw2017simple} has been recently improved in \cite{jeong2020sub}. However, this improved (and optimal for some problems) result only leads to $\kappa^2\log \kappa$ term in \eqref{eq:samplecovariance}.
\item Van Handel \cite{van2017structured} gives an in-expectation version of \eqref{eq:samplecovariance} in the special case where $X$ is a Gaussian random vector. Despite not using Talagrand's majorizing measures theorem, their analysis is based on a Gaussian comparison theorem and does not cover the sub-Gaussian case. Note that in the Gaussian case, the in-expectation bound for the sample covariance matrix can be converted into an optimal high probability bound using one of the special concentration inequalities provided in \cite{koltchinskii2017operators, adamczak2015note, klochkov2020uniform}.
\end{itemize}

Our approach, based on the variational inequality and described in detail in Section \ref{sec:pacbayes}, bypasses several technical steps appearing in the literature. Speaking informally, we use a smoothed version of the $\varepsilon$-net argument that allows to properly capture the complexity of elliptic indexing sets without resorting to generic chaining. This extension will be key to our multilinear results, where the above mentioned tools are hard to apply.

Note that even though Example \ref{ex:samplecovariance} is sharp in the rank one case (see the lower bound in \cite{koltchinskii2017operators}), the result of Theorem \ref{thm:maintheorem}, due to its generality, can be suboptimal in other cases. For example, let $A$ be a diagonal random matrix such that its diagonal elements are the same copy of the absolute value of a standard Gaussian random variable. In this case, Theorem \ref{thm:maintheorem} scales as $
\sqrt{\frac{d + t}{n}},
$
whereas the correct order is
$
\sqrt{\frac{t}{n}}.$ We also remark that, at least in the rank one case, the bound of Theorem \ref{thm:maintheorem} is out of the scope of the so-called matrix concentration inequalities, since they provide additional logarithmic factors and suboptimal tails (see some related bounds in \cite{tropp2012user, Vershynin2016HDP, klochkov2020uniform, lopes2019bootstrapping}). We additionally refer to the recent work \cite{bandeira2021matrix}, where the in-expectation analog of \eqref{eq:samplecovariance} is derived modulo some additional lower-order additive terms.

Motivated by the recent interest in random tensors \cite{vershynin2020concentration, gotze2021concentration, sambale2020some, bubeck2020law, even2021concentration}, we show how our arguments can be extended to provide a multilinear extension of Theorem \ref{thm:maintheorem}. That is, we are considering sums of independent random tensors of order higher than one and want to prove a bound similar to \eqref{eq:samplecovariance}. Let us introduce this setup. Consider the simple (rank one) random symmetric tensor 
\[
X^{\otimes s} = \underbrace{X \otimes \ldots \otimes X}_{s\; \textrm{times}},
\]
where $X$ is a zero mean sub-Gaussian vector in $\mathbb{R}^d$ and $X^{\otimes s}_1, \ldots, X_n^{\otimes s}$ are its independent copies. We are interested in studying 
\begin{equation}
\label{eq:learninglpballs}
\left\|\sum\limits_{i = 1}^{n}(X_i^{\otimes s} - \E X_i^{\otimes s})\right\| = \sup\limits_{v \in S^{d - 1}}\left|\sum\limits_{i = 1}^n(\langle X_i, v\rangle^s - \E\langle X, v\rangle^s)\right|,
\end{equation}
where $\|\cdot\|$ stands for the operator norm of the symmetric $s$-linear form. Here we used that for symmetric forms the expression is maximized by a single vector $v \in S^{d - 1}$ (see e.g., \cite[Section 2.3]{nemirovski2004interior}). 

The question of upper bounding \eqref{eq:learninglpballs} (with $\langle X_i, v\rangle^s$ usually replaced by the absolute value $|\langle X_i, v\rangle|^s$ in the right-hand side and non-integer values of $s$ are allowed) is well studied \cite{giannopoulos2000concentration, guedon2007lp, mendelson2008weakly, adamczak2010quantitative, vershynin2011approximating, mendelson2021approximating}. The results are usually of the following form: Assuming that $X_1, \ldots, X_n$ are i.i.d. copies of the isotropic vector satisfying certain regularity assumptions, one is interested in defining the smallest sample size $n$ such that, with high probability,
\[
\sup\limits_{v \in S^{d - 1}}\left|\frac{1}{n}\sum\limits_{i = 1}^n |\langle X_i, v\rangle|^s - \E|\langle X, v\rangle|^s\right|\le \varepsilon. 
\]
The general form of the assumption (see \cite{guedon2007lp, vershynin2011approximating} and, in particular, \cite[Theorem 4.2]{adamczak2010quantitative}) required to achieve this precision for some regular families of distributions is 
\begin{equation}
\label{eq:conditiononn}
n \ge C_sd^{s/2}\varepsilon^{-2}\log^{2 + s}(2\varepsilon^{-2}),
\end{equation}
where $C_s$ depends only on $s$. Although the condition $n \ge C_sd^{s/2}$ is known to be optimal when $\varepsilon$ is a constant \cite{vershynin2011approximating}, the dependence on $\varepsilon$ is either suboptimal or not explicit in the existing results. In fact, a recent result of Mendelson \cite{mendelson2021approximating} suggests that using a specific robust estimation procedure, it is possible to approximate the moments of the marginals for any $s \ge 1$ with $\varepsilon$ scaling as $\sqrt{\frac{d}{n}\log \frac{n}{d}}$. At the same time,  the inequality \eqref{eq:conditiononn} becomes vacuous in this regime whenever $s > 2$.  Before we proceed, recall the following definition. 

\begin{definition}
The measure $\nu$ in $\mathbb{R}^{d}$ is log-concave, if for any measurable subsets $A, B \in \mathbb{R}^d$ and any $t \in [0, 1]$,
\[
\nu(t A + (1 - t)B) \ge \nu(A)^{t}\nu(B)^{1 - t},
\]
whenever the set $t A + (1 - t)B = \{t a + (1 - t)b : a \in A, b \in B\}$ is measurable. 
\end{definition}

Our next result shows that provided that the sample size $n$ is large enough, one can approximate the $s$-th integer moment of the marginals using 
their empirical counterparts with $\varepsilon$ scaling as $\|\Sigma\|^{s/2}\sqrt{\frac{\mathbf{r}(\Sigma)}{n}}$. This is the best possible approximation rate when $s = 2$ and $X$ is a multivariate Gaussian random vector. Moreover, because of \eqref{eq:conditiononn} this approximation rate was not previously achieved even in the isotropic case.
\begin{theorem}
\label{thm:tensortheorem}
Let $s \ge 2$ be an integer.  Assume that $X_1, \ldots , X_n$ are independent copies of a zero mean vector $X$ that is either sub-Gaussian \eqref{eq:subgauss} or log-concave.  There exist $c_s> 0$ that depends only on $s$ and an absolute constant $c > 0$ such that the following holds. Assume that 
\begin{equation}
\label{eq:assumptiononn}
n \ge c_{s}(\mathbf{r}(\Sigma))^{s - 1}.
\end{equation}
Then, with probability at least $1 - cn\exp(-\sqrt{\mathbf{r}(\Sigma)})$, it holds that
\[
\sup\limits_{v \in S^{d - 1}}\left|\frac{1}{n}\sum\limits_{i = 1}^n\langle X_i, v\rangle^s - \E\langle X, v\rangle^s\right| \le C\|\Sigma\|^{s/2}\sqrt{\frac{\mathbf{r}(\Sigma) }{n}},
\]
where $C = C(s, \kappa)$ in the sub-Gaussian case and $C = C(s)$ in the log-concave case.

Moreover, in the sub-Gaussian case if $n \ge c_{s}(\mathbf{r}(\Sigma))^{s + 1}$, then the same bound holds with probability at least $1 - cn\exp(-\mathbf{r}(\Sigma))$.
\end{theorem}

To simplify our proofs, we focus only on the tensor case; in particular, we consider the integer values of $s$. In Theorem \ref{thm:tensortheorem} we require that either $cn\exp(-\mathbf{r}(\Sigma)) < 1$ or $cn\exp(-\sqrt{\mathbf{r}(\Sigma)}) < 1$. These assumptions can be dropped by slightly inflating our upper bound (see also \cite[Remark 4.3]{adamczak2010quantitative}).
It is likely that in the sub-Gaussian case one can extend our arguments, namely, the decoupling-chaining argument discussed below, so that the assertion holds with probability $1 - cn\exp(-\mathbf{r}(\Sigma))$ whenever $n \ge c_s(\mathbf{r}(\Sigma))^{s - 1}$. Indeed, we know by \eqref{eq:samplecovariance} that this is the case at least when $s = 2$. We preferred a shorter proof instead of a more accurate estimate of the tail.

In the log-concave case when $s = 2$, Theorem \ref{thm:tensortheorem} complements the renowned result of Adamczak, Litvak, Pajor and Tomczak-Jaegermann  \cite{adamczak2010quantitative}. The main advantage of our result is the explicit dependence on the effective rank, similar to the sample covariance bound of Koltchinskii and Lounici in the Gaussian case \cite{koltchinskii2017operators}. Our next result sharpens the tail estimate in this specific case and coincides with the best known bound in the isotropic case. 

\begin{theorem}
\label{thm:logconcavetheorem}
Assume that $X_1, \ldots X_n$ are independent copies of a zero-mean, log-concave random vector $X$ with covariance $\Sigma$. There are absolute constants $c_1, c_2, c_3 > 0$ such that the following holds. We have, with probability at least $1 - c_1\exp(-(\mathbf{r}(\Sigma)n)^{1/4}) - 2\exp(-\mathbf{r}(\Sigma))$,
\[
\left\|\frac{1}{n}\sum\limits_{i = 1}^nX_iX_i^{\top} - \Sigma\right\| \le c_2\|\Sigma\|\sqrt{\frac{\mathbf{r}(\Sigma)}{n}},
\]
whenever $n \ge c_3\mathbf{r}(\Sigma)$.
\end{theorem}

In both proofs we combine the variational inequality approach with the decoupling-chaining argument developed in \cite{adamczak2010quantitative, Talagrand2014}. We remark that the adaptation of the latter argument to the non-isotropic case is quite straightforward. Our analysis improves the analysis of the \emph{spread part} \cite[Section 9.4]{Talagrand2014} whose analysis in the isotropic case combines the $\varepsilon$-net argument and the Bernstein inequality. Finally, observe that since the Gaussian distribution is log-concave, the rate of convergence $\|\Sigma\|\sqrt{\frac{\mathbf{r}(\Sigma)}{n}}$ is the best possible due to the lower bound in \cite{koltchinskii2017operators}.

There is a version of Theorem \ref{thm:logconcavetheorem}, which follows from our proof with minimal changes. For the reader's convenience we also present an explicit tail bound.

\begin{proposition}
\label{cor:logconcavetheorem}
Assume that $X_1, \ldots X_n$ are independent copies of a zero-mean random vector with covariance $\Sigma$ such that for some $\kappa \ge 1$ and all $y \in \mathbb{R}^d$, it holds that
\[
\|\langle X, y \rangle\|_{\psi_1} \le \kappa\sqrt{y^{\top}\Sigma y} \quad \textrm{and}\quad \max\limits_{i}\|X_i\| \le \kappa\sqrt{\|\Sigma\|}(\mathbf{r}(\Sigma)n)^{\frac{1}{4}} \quad \textrm{almost surely}.
\]
There are absolute constants $c_1, c_2, c_3 > 0$ such that the following holds. For any $t \ge 0$ we have, with probability at least $1 - c_1\exp(-t)$, 
\[
\left\|\frac{1}{n}\sum\limits_{i = 1}^nX_iX_i^{\top} - \Sigma\right\| \le c_2\kappa^2\|\Sigma\|\left(\sqrt{\frac{\mathbf{r}(\Sigma) + t}{n}} + \frac{t^2}{n}\right),
\]
whenever $n \ge c_3(\mathbf{r}(\Sigma) + t)$.
\end{proposition}

In Section \ref{sec:additionalresults} we provide two additional results: a bound on the deviation of the norm of a sub-exponential random vector and a lower tail version of Theorem \ref{thm:maintheorem}.
Finally, as a part of the proof of Theorem \ref{thm:tensortheorem}, we provide a simple proof of the bound by Hsu, Kakade and Zhang \cite{hsu2012tail} on the deviation of the norm of a sub-Gaussian random vector.

\section{An approach based on the variational equality}
\label{sec:pacbayes}
Our approach will be based on the following duality relation (see \cite[Corollary 4.14]{boucheron2013concentration}): for a probability space $(\Theta, \mu)$ and any measurable function $g$ such that $\E_\mu\exp(g(\theta)) < \infty$, it holds that
\begin{equation}
\label{eq:variationalformula}
\log\E_{\mu}\exp(g(\theta)) = \sup\limits_{\rho \ll \mu}(\E_{\rho}g(\theta) - \mathcal{KL}(\rho, \mu)),
\end{equation}
where the supremum is taken with respect to all measures absolutely continuous with respect to $\mu$ and
\[
\mathcal{KL}(\rho, \mu) = \int\log\left(\frac{d\rho}{d\mu}\right)d\rho
\]
denotes the Kullback-Leibler divergence between $\rho$ and $\mu$. The equality \eqref{eq:variationalformula} is used in the proof of the additivity of entropy \cite[Proposition 5.6]{Ledoux2001} and in the transportation method for proving concentration inequalities \cite[Chapter 8]{boucheron2013concentration}. A useful corollary of the variational equality is the following lemma (see e.g., \cite[Proposition 2.1]{catoni2017dimension} and discussions therein).
\begin{lemma}
\label{lem:pacbayes}
Assume that $X_i$ are i.i.d. random variables defined on some measurable space $\mathcal X$. Assume also that $\Theta$ (called the parameter space) is a subset of $\mathbb{R}^p$. Let $f: \mathcal X \times \Theta \to \mathbb{R}$ be such that $\E_X\exp(f(X, \theta)) < \infty$ almost surely. Let $\mu$ be a distribution (called prior) on $\Theta$ and let $\rho$ be any distribution (called posterior) on $\Theta$ such that $\rho \ll \mu$. Then, with probability at least $1 - \exp(-t)$, simultaneously for all such $\rho$ we have
\[
\frac{1}{n}\sum\limits_{i = 1}^n\E_{\rho}f(X_i, \theta) \le \E_{\rho}\log(\E_X\exp(f(X, \theta))) + \frac{\mathcal{KL}(\rho, \mu) + t}{n},
\]
where $\theta$ is distributed according to $\rho$. 
\end{lemma}
\begin{proof}
We sketch the proof for the sake of completeness. Let $g: \mathcal X^n \times \Theta \to \mathbb{R}$ in \eqref{eq:variationalformula} be equal to $\sum\limits_{i = 1}^nf(X_i, \theta) - n\log\E_X \exp (f(X, \theta))$. Let $\E$ denote the expectation with respect to the i.i.d. sample $X_1, \ldots, X_n$. Using successively \eqref{eq:variationalformula}, Fubini's theorem and independence of $X_1, \ldots, X_n$, we have
\begin{align*}
&\E\exp\sup\limits_{\rho \ll \mu}\left(\E_{\rho}\sum_{i = 1}^nf(X_i, \theta) -  \E_{\rho} n \log \E_X\exp(f(X, \theta)) - \mathcal{KL}(\rho, \mu)\right)
\\
&\qquad = \E\E_{\mu}\exp\left(\sum_{i = 1}^nf(X_i, \theta) - n \log \E_X\exp(f(X, \theta)) \right) 
\\
&\qquad = \E_{\mu}(\E_X\exp\left(f(X, \theta) - \log \E_X\exp(f(X, \theta)) \right)^n = 1.
\end{align*}
By Markov's inequality for any random variable $Y$ the identity $\E\exp(Y) = 1$ implies that $Y < t$, with probability at least $1 - \exp(-t)$. The claim follows by taking
\[
Y = \sup\limits_{\rho \ll \mu}\left(\E_{\rho}\sum_{i = 1}^nf(X_i, \theta) -  \E_{\rho} n \log \E_X\exp(f(X, \theta)) - \mathcal{KL}(\rho, \mu)\right).
\]
\end{proof}

\begin{remark}
In Lemma \ref{lem:pacbayes} we assumed that $\E_X\exp(f(X, \theta)) < \infty$ for all $\theta \in \Theta$. However, this does not imply that $f(X, \theta)$ is integrable with respect to $\rho$. If it is not the case, one can conventionally take the cases where $\E_{\rho}f(X, \theta)$ is infinite into account, so that the inequality of Lemma \ref{lem:pacbayes} still holds.  For more details see \cite[Appendix A]{catoni2016pac}.
\end{remark}

Our analysis is inspired by the application of \eqref{eq:variationalformula} and Lemma \ref{lem:pacbayes} in the works of Catoni and co-authors \cite{audibert2011robust, audibert2010linear, catoni2012challenging, catoni2016pac, catoni2017dimension} on robust mean and covariance estimation as well as by the work of Oliveira \cite{oliveira2016lower} on the lower tails of sample covariance matrices under minimal assumptions. This approach is usually called the PAC-Bayesian method in the literature. In robust mean estimation, one is making minimal distributional assumptions (for example, by considering heavy-tailed distributions) aiming to estimate the mean of the random variable/vector/matrix using the estimators that necessarily differ from the sample mean (see \cite{catoni2012challenging, minsker2015geometric, lugosi2019sub, catoni2017dimension, mendelson2020robust, chinot2019robust, ostrovskii2019affine, klochkov2020robust, mendelson2021approximating} and the recent survey \cite{lugosi2019mean}). Our aim is somewhat different:  We work with sums of independent random matrices and multilinear forms. It is important to note that except the recent works of Catoni and Giulini \cite{giulini2018robust, catoni2017dimension}, statistical guarantees based on \eqref{eq:variationalformula} are dimension-dependent.  A detailed technical comparison with these papers is deferred to Section \ref{sec:detailedcomparison}.

\subsection{Motivating examples: matrices with isotropic sub-Gaussian rows and the Gaussian complexity of ellipsoids}
To motivate (and illustrate) the application of the variational equality \eqref{eq:variationalformula} in the context of high-dimensional probability, we first show how Lemma \ref{lem:pacbayes} can be used to recover the standard  bound on the largest and smallest singular values of the the $n$ by $d$ random matrix $A$ having independent, mean zero, isotropic ($\E A_iA_i^{\top} = I_d$) sub-Gaussian rows. In this case, \eqref{eq:subgauss} can be rewritten as 
\[
\sup_{y \in S^{d - 1}}\|\langle y, A_i\rangle\|_{\psi_2} \le \kappa.
\]
In view of \cite[Lemma 4.1.5]{Vershynin2016HDP}, it is enough to show the following statement.
\begin{proposition}(\cite[Theorem 4.6.1]{Vershynin2016HDP})
\label{prop:insteadofunionbound}
Let $A$ be an $n$ by $d$ random matrix whose rows $A_i^{\top}$ are independent, mean zero, sub-gaussian isotropic random vectors. We have for any $t \ge 0$, with probability at least $1 - \exp(-t)$,
\[
\left\|\frac{1}{n}\sum\limits_{i = 1}^nA_iA_i^{\top} - I_d\right\| \le 52\kappa^2\sqrt{\frac{d + t}{n}},
\]
whenever $n \ge d+ t$.
\end{proposition}

The standard way of proving Proposition \ref{prop:insteadofunionbound} uses an $\varepsilon$-net argument combined with the Bernstein inequality in terms of the $\|\cdot\|_{\psi_1}$ norm and the union bound. We demonstrate that if the prior $\mu$ and the posterior $\rho$ are correctly chosen, then Lemma \ref{lem:pacbayes} recovers the same bound without directly exploiting a discretization argument. \begin{remark}
The condition $n \ge d + t$ does not appear in \cite[Theorem 4.6.1]{Vershynin2016HDP}. As a result, this bound contains an additional additive term scaling as $\frac{d +t}{n}$. In our regime when $n \ge d + t$, this term is naturally dominated by $\sqrt{\frac{d + t}{n}}$. In some sense, we only captured the sub-Gaussian regime in the deviation bound. This regime is arguably the most interesting when considering statistical estimation problems.
\end{remark}

Our analysis requires the following standard result. Since we need a version with an explicit constant, we reproduce these lines for the sake of completeness. 
\begin{lemma}
\label{lem:psionenorm}
Let $Y$ be a zero mean random variable. Then for any $\lambda$ such that $|\lambda| \le \frac{1}{2\|Y\|_{\psi_1}}$,
\[
\E\exp(\lambda Y) \le \exp(4\lambda^2\|Y\|_{\psi_1}^2).
\]
\end{lemma}
\begin{proof}
First, by Markov's inequality and any $t \ge 0$, it holds that
\[
\Pr(|Y| \ge t) = \Pr\left(\exp(|Y|/\|Y\|_{\psi_1}) \ge \exp(t/\|Y\|_{\psi_1})\right) \le 2\exp\left(-t/\|Y\|_{\psi_1}\right).
\]
In the following lines, we assume without loss of generality that $\|Y\|_{\psi_1} = 1$. We have for $p \ge 1$,
\begin{equation}
\label{eq:psione}
\E|Y|^p = \int\limits_{t = 0}^{\infty}\Pr(|Y|^p \ge t)dt = \int\limits_{t = 0}^{\infty}\Pr(|Y| \ge t)pt^{p-1}dt \le 2\int\limits_{t = 0}^{\infty}\exp(-t)pt^{p-1}dt = 2p\Gamma(p) = 2p!\;.
\end{equation}
Finally, when $|\lambda| \le 1/2$ by Taylor's expansion and since $\E Y = 0$, we have
\[
\E\exp(\lambda Y) = 1 + \sum\limits_{p = 2}^\infty\frac{\lambda^p\E|Y|^p}{p!} \le 1 + 2\E\sum\limits_{p = 2}^\infty\lambda^p = 1 + \frac{2\lambda^2}{1 - \lambda} \le 1 + 4\lambda^2 \le \exp(4\lambda^2).
\]
The claim follows.
\end{proof}

\begin{proof}(of Proposition \ref{prop:insteadofunionbound})
Fix $\varepsilon > 0$. Our aim is to choose $\mu$ and $\rho$. Let 
\[
\Theta = (1 + \varepsilon)B^d\times (1 + \varepsilon)B^d,
\]
where $B^d$ is a unit ball in $\mathbb{R}^d$. Choose $\mu$ to be a product of two uniform measures each defined on $(1 + \varepsilon)B^d$. For $u, v\in S^{d - 1}$ let $\rho_{u, v}$ be a product of uniform distributions on the balls $\{x \in \mathbb{R}^{d}:\|x - u\| \le \varepsilon\}$ and $\{x \in \mathbb{R}^{d}:\|x - v\| \le \varepsilon\}$. Observe that both balls belong to $(1 + \varepsilon)B^d$. Because of this, if $(\theta, \nu)$ is distributed according to $\rho_{u, v}$, we have $\E_{\rho_{u, v}}(\theta, \nu) = (u, v)$. By the additivity of $\mathcal{KL}$-divergence for product measures and the formula for the volume of the $d$-dimensional ball, we have
\[
\mathcal{KL}(\rho_{u, v}, \mu) = 2\int_{\|x - u\| \le \varepsilon}\log\left(\frac{(1 + \varepsilon)^d}{\varepsilon^d}\right)\frac{1}{\operatorname{vol}(\{x \in \mathbb{R}^{d}:\|x - v\| \le \varepsilon\})}dx = 2d\log\left(\frac{1 + \varepsilon}{\varepsilon}\right),
\]
where $\operatorname{vol}(S)$ denotes the volume of the set $S$. 
Fix $\lambda \in \mathbb{R}$ and consider the random variable $\lambda \theta^{\top}A_iA_i^{\top}\nu$, where $(\theta, \nu)$ is distributed according to $\rho_{u, v}$. We want to plug this random variable into Lemma \ref{lem:pacbayes}. Observe that conditionally on $(\theta, \nu)$, we have, using $\|YZ\|_{\psi_1} \le \|Y\|_{\psi_2}\|Z\|_{\psi_2}$, 
\[
\|\theta^{\top}A_iA_i^{\top}\nu\|_{\psi_1} \le \|\theta\|\|\nu\|\sup_{y \in S^{d - 1}}\|\langle y, A_i\rangle\|^2_{\psi_2} \le (1 + \varepsilon)^2\kappa^2,
\]
where the last inequality follows from the fact that $\theta, \nu \in (1 + \varepsilon)B^d$ almost surely.
Conditionally on $(\theta, \nu)$, combining the triangle and Jensen's inequalities, we have
\[
\|\theta^{\top}A_iA_i^{\top}\nu - \theta^{\top}\E A_iA_i^{\top}\nu\|_{\psi_1} \le  2\|\theta^{\top}A_iA_i^{\top}\nu\|_{\psi_1}.
\]
Further, since $\E A_iA_i^{\top} = I_d$ we have by Lemma \ref{lem:psionenorm}, conditionally on $(\theta, \nu)$,
\[
\E \exp(\lambda\theta^{\top}A_iA_i^{\top}\nu) \le \exp(\lambda\theta^{\top}\nu + 16\lambda^2(1 + \varepsilon)^4\kappa^4),
\]
whenever $\lambda \le \frac{1}{4(1 + \varepsilon)^2\kappa^2}$.
Therefore, since $\E_{\rho_{u, v}} \theta^{\top}A_iA_i^{\top}\nu = u^{\top}A_iA_i^{\top}v$ and $\E_{\rho_{u, v}} \theta^{\top}\nu = u^{\top}v$ Lemma \ref{lem:pacbayes} gives that simultaneously for all $u, v \in S^{d - 1}$, with probability at least $1 - \exp(-t)$,
\begin{equation}
\label{eq:singvalues}
\frac{1}{n}\sum\limits_{i = 1}^n \lambda u^{\top}A_iA_i^{\top}v - \lambda u^{\top}v \le 16\lambda^2\kappa^4(1 + \varepsilon)^4 + \frac{2d\log\left((1 + \varepsilon)/\varepsilon\right) + t}{n}.
\end{equation}
We choose $\varepsilon = \frac{1}{\sqrt{e} - 1}$ to guarantee that $2d\log\left((1 + \varepsilon)/\varepsilon\right) = d$.
Then, taking $\lambda = \frac{1}{4(1 + \varepsilon)^2\kappa^2}\sqrt{\frac{d + t}{n}}$, we require $n \ge d + t$. Simplifying \eqref{eq:singvalues} for this choice of parameters, we prove the claim.
\end{proof}

Another motivating fact is that \eqref{eq:variationalformula} correctly reflects the Gaussian complexity of the ellipsoid. It is well-known that for the ellipsoids, the Dudley integral argument does not give an optimal bound, while the generic chaining does (see \cite[Chapter 2.5]{Talagrand2014}). Although one can instead directly use the Cauchy-Schwarz inequality, it is easy to show that the variational equality \eqref{eq:variationalformula} captures the same bound. 

\begin{example}[The Gaussian complexity of ellipsoids via the variational equality]
\label{ex:ellipsoids}
Let $Z$ be a standard normal random vector. Let $\Sigma$ be a positive semi-definite $d$ by $d$ matrix. It hold that
\[
\E \sup\limits_{v \in \Sigma^{1/2}S^{d-1}} \langle Z, v \rangle \le \sqrt{\tr(\Sigma)}. 
\]
\end{example}

\begin{proof}
Set  $\Theta = \mathbb{R}^d$ and let $\beta > 0$.
Let the prior distribution $\mu$ be a multivariate Gaussian distribution with mean zero and covariance $\beta^{-1}\Sigma$. For $v \in \Sigma^{1/2} S^{d - 1}$ let the distribution $\rho_v$ be a multivariate Gaussian distribution with mean $v$ and covariance $\beta^{-1}\Sigma$. By the standard formula, we have 
\[
\mathcal{KL}(\rho_v, \mu) = \beta/2.
\]
Let $\theta$ be distributed according to $\rho_v$. By Jensen's inequality for any $\lambda > 0$,
\[
\E \sup\limits_{v \in \Sigma^{1/2}S^{d-1}} \langle Z, v \rangle = \E \sup\limits_{v \in \Sigma^{1/2}S^{d-1}}\E_{\rho_v}\langle Z, \theta \rangle \le \frac{1}{\lambda}\log\E \sup\limits_{v \in \Sigma^{1/2}S^{d-1}} \exp(\lambda\E_{\rho_v}\langle Z, \theta\rangle).
\]
By the line of the proof of Lemma \ref{lem:pacbayes}, we have
\[
\E\sup\limits_{v \in \Sigma^{1/2}S^{d - 1}}\exp\left(\E_{\rho_v}\lambda \langle Z, \theta \rangle -  \E_{\rho_v}\log \E\exp(\lambda \langle Z, \theta \rangle) - \beta/2\right) \le 1.
\]
Since $Z$ is a standard normal random vector, it holds that
\[
\E_{\rho_v}\log \E\exp(\lambda \langle Z, \theta \rangle) = \E_{\rho_v}\lambda^2\|\theta\|^2/2 = \lambda^2(\|v\|^2 + \beta^{-1}\tr(\Sigma))/2.
\]
Combining previous inequalities and simplifying, we have
\[
\E \sup\limits_{v \in \Sigma^{1/2}S^{d-1}} \langle Z, v \rangle \le \inf\limits_{\lambda, \beta > 0} \lambda(\|\Sigma\| + \beta^{-1}\tr(\Sigma))/2 + \beta/(2\lambda) = \sqrt{\tr(\Sigma)}.
\]
The claim follows.
\end{proof}

Observe that the proof explicitly uses a bound on the expected squared norm of a multivariate normal vector (not for the norm of $Z$ though), which is closer to a more \say{algebraic} approach based on the Cauchy-Schwarz inequality, whereas the generic chaining is a \say{geometric} approach; we refer to \cite[Chapter 2.5]{Talagrand2014} for a detailed discussion of the Gaussian complexity of ellipsoids.

\subsection{Proof of Theorem \ref{thm:maintheorem}}

In view of Proposition \ref{prop:insteadofunionbound} and  Example \ref{ex:ellipsoids}, a natural idea is to use the uniform distribution for $\rho$ and $\mu$ on ellipsoids induced by the structure of the matrix $\Sigma$. It appears that working with ellipsoids directly is quite tedious. To avoid these technical problems, we work with the non-isotropic truncated Gaussian distribution. Throughout the proof, we assume without loss of generality that $\Sigma$ is invertible. If it is not the case, the distribution of $M$ lives almost surely in a lower-dimensional subspace. We can project on this subspace and continue the proof without changes. Fix $\beta > 0$. Let 
\[
\Theta = \mathbb{R}^d\times \mathbb{R}^d,
\]
and choose the prior distribution $\mu$ on $\Theta$ as the product of two multivariate Gaussian distributions in $\mathbb{R}^d$ both with mean zero and covariance matrix $\beta^{-1}\Sigma$.
For $u, v \in \Sigma^{1/2}S^{d - 1}$ let the posterior distribution $\rho_{u, v}$ be defined as follows. For $r > 0$ consider the density function $f_u$ in $\mathbb{R}^d$ given by
\begin{equation}
\label{eq:density}
f_u(x) = \frac{1}{p(2\pi)^{d/2}\sqrt{\det(\beta^{-1}\Sigma)}}\exp\left(-\frac{(x - u)^{\top}\beta\Sigma^{-1}(x - u)}{2}\right)\ind{\|x - u\| \le r},
\end{equation}
where $p > 0$ is a normalization constant. That is, the distribution defined by $f_u$ is a multivariate normal distribution restricted to the ball $\{x \in \mathbb{R}^d: \|x - u\| \le r\}$. Our distribution $\rho_{u, v}$ on $\Theta$ is now defined as a product of two distributions given by $f_u$ and $f_v$ respectively. Observe that since $f_u$ is symmetric around $u$ (that is, for any $y \in \mathbb{R}^d$, we have $f_{u}(u + y) = f_{u}(u - y)$), we have for $(\theta, \nu)$ distributed according to $\rho_{u, v}$,
\begin{equation}
\label{eq:expectation}
\E_{\rho_{u, v}}(\theta, \nu) = (\E_{\rho_u}\theta, \E_{\rho_v}\nu) = (u, v),
\end{equation}
where $\rho_u$ and $\rho_v$ denote the marginals of $\rho_{u, v}$.

Let us now compute the Kullback-Leibler divergence between $\rho_{u, v}$ and $\mu$. Let $g$ denote the density function of a multivariate Gaussian distribution with mean zero and covariance $\beta^{-1}\Sigma$. By the additivity of the Kullback-Leibler divergence  for product measures, we have
\[
\mathcal{KL}(\rho_{u, v}, \mu) =  \int\log\left(\frac{f_u(x)}{g(x)}\right)f_u(x)dx+ \int\log\left(\frac{f_v(x)}{g(x)}\right)f_v(x)dx.
\]
Both terms are now analyzed similarly. For $\theta$ distributed according to $\rho_{u}$,
\begin{align*}
\int\log\left(\frac{f_u(x)}{g(x)}\right)f_u(x)dx &= \E_{\rho_{u}}\log\left(\frac{1}{p}\exp\left(\frac{-(\theta - u)^{\top}\beta\Sigma^{-1}(\theta - u) + \theta^{\top}\beta\Sigma^{-1}\theta}{2}\right)\right)
\\
&=\log\left(\frac{1}{p}\right) + \E_{\rho_{u}}\left(\frac{u^{\top}\beta\Sigma^{-1}\theta + \theta^{\top}\beta\Sigma^{-1}u - u^{\top}\beta\Sigma^{-1}u}{2}\right)
\\
&= \log\left(\frac{1}{p}\right) + \frac{u^{\top}\beta\Sigma^{-1}u}{2} = \log\left(\frac{1}{p}\right) + \frac{\beta}{2},
\end{align*}
where in the last line we used $u \in \Sigma^{1/2}S^{d - 1}$.
Let $Z$ be a random vector having a multivariate Gaussian distribution with mean zero and covariance $\beta^{-1}\Sigma$. By \eqref{eq:density} and using the translation $u \to 0$, we have 
$
p = \Pr(\|Z\| \le r).
$
By Markov's inequality we have $\Pr(\|Z\| > r) \le \E\|Z\|^2/r^2 = \beta^{-1}\tr(\Sigma)/r^2$. We choose 
\[
r = \sqrt{2\beta^{-1}\tr(\Sigma)}
\]
and get $p \ge 1/2$. Therefore, we have $\log\left(1/p\right) \le \log 2$. Finally, for this choice of $r$,\begin{equation}
\label{kl:boundeq}
\mathcal{KL}(\rho_{u, v}, \mu) \le 2\log 2 + \beta,
\end{equation}
For $\lambda \in \mathbb{R}$ we want to plug the function $\lambda \theta^{\top} \Sigma^{-1/2}M\Sigma^{-1/2}\nu$ into Lemma \ref{lem:pacbayes}, where $(\theta, \nu)$ is distributed according to $\rho_{u, v}$. By \eqref{eq:expectation} we have
\begin{equation}
\label{eq:emppart}
\frac{1}{n}\sum\limits_{i = 1}^n\E_{\rho_{u, v}} \lambda \theta^{\top} M_i\nu = \frac{1}{n}\sum\limits_{i = 1}^n\lambda u^{\top} \Sigma^{-1/2}M_i\Sigma^{-1/2}v.
\end{equation}
It is only left to compute
$
\E_{\rho_{u, v}}\log(\E\exp(\lambda\theta^{\top} \Sigma^{-1/2}M\Sigma^{-1/2}\nu)).
$
Conditionally on $(\theta, \nu)$, we have as in the proof of Proposition \ref{prop:insteadofunionbound}
\begin{align*}
&\|\theta^{\top} \Sigma^{-1/2}M\Sigma^{-1/2}\nu - \theta^{\top} \Sigma^{-1/2}\E M\Sigma^{-1/2}\nu\|_{\psi_1}
\\
&\quad\le 2\|\theta^{\top} \Sigma^{-1/2}M\Sigma^{-1/2}\nu\|_{\psi_1}
\\
&\quad\le \|\theta^{\top} \Sigma^{-1/2}M\Sigma^{-1/2}\theta\|_{\psi_1} + \|\nu^{\top} \Sigma^{-1/2}M\Sigma^{-1/2}\nu\|_{\psi_1}
\\
&\quad\le \kappa^2(\|\theta\|^2 + \|\nu\|^2),
\end{align*}
where both the $\|\cdot\|_{\psi_1}$ norm and the expectation are considered with respect to the distribution of $M$, and the second line uses the Cauchy-Schwarz inequality. Taking again the expectation with respect to $M$ only, we have by Lemma \ref{lem:psionenorm}
\[
\E\exp(\lambda\theta^{\top} \Sigma^{-1/2}M\Sigma^{-1/2}\nu) \le \exp(\lambda\theta^{\top}\Sigma^{-1/2}\E M\Sigma^{-1/2}\nu)\exp(4\lambda^2\kappa^4(\|\theta\|^2+ \|\nu\|^2)^2),
\]
provided that $|\lambda| \le \frac{1}{2\kappa^2(\|\theta\|^2 + \|\nu\|^2)}$. Observe that by our choice of $r$, we have almost surely 
\[
\max\{\|\theta\|^2, \|\nu\|^2\} \le (\sqrt{\|\Sigma\|} + r)^2 = (\sqrt{\|\Sigma\|} + \sqrt{2\beta^{-1}\tr(\Sigma)})^2.
\]
Let us choose $\beta = 2\mathbf{r}(\Sigma)$. Thus, by \eqref{kl:boundeq}, \eqref{eq:emppart} and Lemma \ref{lem:pacbayes}
we have for any fixed $\lambda$ such that $|\lambda| \le \frac{1}{16\kappa^2\|\Sigma\|}$ simultaneously for all $u, v \in S^{d - 1}$,
\[
\frac{1}{n}\sum\limits_{i = 1}^n\lambda u^{\top} Mv \le \lambda u^{\top} \Sigma v + 64\lambda^2\kappa^4\|\Sigma\|^2 + \frac{4\mathbf{r}(\Sigma) + t}{n},
\]
where we used $2\log 2 + 2\mathbf{r}(\Sigma) \le 4\mathbf{r}(\Sigma)$.
We choose $\lambda = \frac{1}{16\kappa^2\|\Sigma\|}\sqrt{\frac{4\mathbf{r}(\Sigma)+t}{n}}$ and finish the proof.
\qed

\subsection{Proofs of Theorem \ref{thm:tensortheorem} and Theorem \ref{thm:logconcavetheorem}}
\label{sec:tensors}

We first present some auxiliary results, then we prove Theorem \ref{thm:logconcavetheorem} and Theorem \ref{thm:tensortheorem}. The technique of the proof combines the analysis of Theorem \ref{thm:maintheorem} with a careful truncation argument. We also use the decoupling-chaining argument to control the large components in the sums. In the last part, we are mainly adapting the previously known techniques.

We need the following result, which is similar to the deviation inequality appearing in \cite{hsu2012tail}. As above, we provide a simple proof based on the variational equality \eqref{eq:variationalformula}.

\begin{lemma}
\label{lem:concentrationofthenorm}
Assume that $X$ is a zero mean $\kappa$-sub-Gaussian random vector \eqref{eq:subgauss}.
Then, with probability at least $1 - \exp(-t)$,
\begin{equation}
\label{eq:concentrationofthenorm}
\|X\|^2 \le 36\kappa^2\left(\tr(\Sigma)/2 + \sqrt{2t\tr(\Sigma)\|\Sigma\|} + t\|\Sigma\|\right).
\end{equation}
\end{lemma} 
\begin{remark}
We will be frequently using the following relaxation of the bounds \eqref{eq:concentrationofthenorm}:
\[
\|X\|^2 \le 36\kappa^2\left(\tr(\Sigma) + 2t\|\Sigma\|\right).
\]
\end{remark}
\begin{proof}
Observe that $\|X\| = \sup_{v \in S^{d - 1}}\langle X, v\rangle$. Thus, we upper bound $\langle X, v\rangle$ uniformly over the sphere.
Set $\Theta = \mathbb{R}^d$. Let the prior distribution $\mu$ be a multivariate Gaussian distribution with mean zero and covariance $\beta^{-1}I_d$. For $v \in S^{d - 1}$ let $\rho_v$ be a multivariate Gaussian distribution with mean $v$ and covariance $\beta^{-1}I_d$. By the standard formula, we have 
\[
\mathcal{KL}(\rho_v, \mu) = \beta/2.
\]
Our function is $\lambda \langle X,  \theta\rangle$, where $\theta$ is distributed according to $\rho_v$. To apply Lemma \ref{lem:pacbayes} (with $n = 1$) we only need to compute 
$
\E_{\rho_v}\log(\E_X\exp(\lambda \langle X, \theta\rangle)).
$
Conditionally on $\theta$, by the sub-Gaussian assumption, we have 
\begin{equation}
\label{eq:expmom}
\log \E_X\exp(\lambda \langle X,  \theta\rangle) \le 9\kappa^2\lambda^2\theta^{\top}\Sigma \theta,
\end{equation}
where to get the explicit constant, one should keep track of the constant factors in the implications of \cite[Proposition 2.5.2]{Vershynin2016HDP}. We have
\[
\E_{\rho_v}(\theta^{\top} \Sigma \theta) = v^{\top}\Sigma v + \beta^{-1}\tr(\Sigma).
\]
Therefore, for any $\lambda > 0$, simultaneously for all $v \in S^{d - 1}$, we have, with probability at least $1 - \exp(-t)$,
\[
\langle X, v\rangle \le 9\kappa^2\lambda\left(\beta^{-1}\tr(\Sigma) + \|\Sigma\|\right) + \lambda^{-1}(\beta/2 + t).
\]
Choosing $\lambda = \sqrt{\frac{\beta/2 + t}{9\kappa^2(\beta^{-1}\tr(\Sigma) + \|\Sigma\|)}}$ and $\beta = \sqrt{\frac{2t\tr(\Sigma)}{\|\Sigma\|}}$, we prove the claim.
\end{proof}

\begin{remark}
The leading constant in Lemma \eqref{lem:concentrationofthenorm} can be made optimal if we assume a bound on the moment generating function as in the Gaussian case. That is, if instead of \eqref{eq:expmom}, we have, conditionally on $\theta$,
\[
\log \E_X\exp(\lambda \langle X,  \theta\rangle) \le \lambda^2\theta^{\top}\Sigma \theta/2,
\]
then, optimizing with respect to $\lambda$ and $\beta$, one can show that, with probability at least $1 - \exp(-t)$, it holds that 
\[
\|X\|^2 \le \tr(\Sigma) + 2\sqrt{2t\tr(\Sigma)\|\Sigma\|} + 2t\|\Sigma\|.
\]
This is what one can achieve if $X$ is a zero mean Gaussian vector, in which case one can use the Gaussian concentration inequality (see \cite[Example 5.7]{boucheron2013concentration}). This observation appears (implicitly) in the works of Catoni and co-authors.
\end{remark}
Define the truncation function
\begin{equation}
\psi(x) = 
    \begin{cases}
      x,\quad &\textrm{for}\; x \in [-1, 1];
      \\
      \operatorname{sign}(x),\quad &\textrm{for}\; |x| > 1.
    \end{cases}
\end{equation}
It is a symmetric function such that for all $x \in \mathbb{R}$, 
\[
\psi(x)\le \log(1 + x + x^2).
\]
The following bound will be important in our analysis.

\begin{lemma}
\label{lem:almostconvex}
Let $\psi$ be as above and let $Z$ be a square integrable random variable. We have
\[
\psi(\E Z) \le \E\log(1 + Z + Z^2) + \min\{1, \E Z^2/6\}.
\]
Moreover, for any $a > 0$, it holds that
\[
 \E\log(1 + Z + Z^2) + a\E\min\{1, Z^2/6\} \le \E\log\left(1 + Z + \left(1 + \frac{(7 + \sqrt{6})(\exp(a) - 1)}{6}\right) Z^2\right).
\]
\end{lemma}

\begin{proof}
In the proof we use the following fact. The function
\[
x \mapsto \log(1 + x + x^2) + x^2/6
\]
is convex. Observe also that if $0 \le t \le a$, then
\begin{equation}
\label{eq:useful}
    \exp(t) \le 1 + \frac{t(e^{a} - 1)}{a}.
\end{equation}
We proceed with the following lines
\begin{align*}
    \psi(\E Z) &= \min\{\psi(\E Z), 1\}
    \\
    &\le \min\{\psi(\E Z) + (\E Z)^2/6, 1\}
    \\
    &\le \min\{\log(1 + \E Z + (\E Z)^2) + (\E Z)^2/6, 1\}
    \\
    &\le \min\{\E \left(\log(1 + Z + Z^2) + Z^2/6\right), 1\}\quad(\text{By Jensen's inequality})
    \\
    &\le \E \log(1 + Z + Z^2) + \min\{1, \E Z^2/6\}.
 \end{align*}
For the second inequality we have
 \begin{align*}
    &\E\log(1 + Z + Z^2) + a\E\min\{1, Z^2/6\} 
    \\
    &= \E \log\left(\left(1 + Z + Z^2\right)\exp(\min\{a, aZ^2/6\})\right) 
    \\
    &\le\E \log\left(\left(1 + Z + Z^2\right)\left(1 + \min\{1, Z^2/6\}\left(e^{a} - 1\right)\right)\right)\quad \text{by}\; \eqref{eq:useful}.
\end{align*}
Further, we have
\begin{align*}
(1 + Z + Z^2)\min\{1, Z^2/6\} &\le (Z^2/6 + Z^2 + Z\min\{1, Z^2/6\})
\\
&\le (7/6 + 1/\sqrt{6})Z^2, 
\end{align*}
where we used that if $|Z| \le \sqrt{6}$, then $
Z \min\{1, Z^2/6\} \le |Z|^3/6 \le Z^2/\sqrt{6}.$ Otherwise, if $|Z| \ge \sqrt{6}$, then
$
Z \le Z^2/\sqrt{6}.
$
These computations conclude the proof.
\end{proof}

This allows us to prove the following uniform bound.  
\begin{lemma}
\label{lem:truncationlemma}
Assume that $X_1, \ldots X_n$ are independent copies of a zero-mean random vector $X$ in $\mathbb{R}^d$ with covariance $\Sigma$. Let $s$ be an integer and assume that $X$ satisfies that for some $\eta \ge 1$ and any $y \in \mathbb{R}^d$,
 \begin{equation}
 \label{eq:subexponentialassumption}
 \left(\E|\langle y, X\rangle|^{2s}\right)^{\frac{1}{2s}} \le \eta\sqrt{y^{\top}\Sigma y}.
 \end{equation}
 Fix the truncation level 
 \[
 \lambda = \sqrt{\frac{\mathbf{r}(\Sigma) + t}{n\eta^{2s}\|\Sigma\|^s}}.
 \]
Then there is $C_s > 0$ that depends only on $s$ such that for any $t > 0$, with probability at least $1 - 2\exp(-t)$, it holds that
\[
\sup\limits_{v \in S^{d - 1}}\left|\frac{1}{n\lambda}\sum\limits_{i = 1}^n\psi(\lambda\langle v, X_i\rangle^s) - \E\langle v, X\rangle^s\right| \le C_s\eta^s\|\Sigma\|^{s/2}\sqrt{\frac{\mathbf{r}(\Sigma) + t}{n}}.
\]
\end{lemma}

\begin{proof}
Fix $\beta > 0$ and let $
\Theta = (\mathbb{R}^d)^s.$
Choose the prior distribution $\mu$ on $\Theta$ as the product of $s$ multivariate Gaussian distributions in $\mathbb{R}^d$ with mean zero and covariance $\beta^{-1}I_d$.  For $v \in S^{d - 1}$ let the posterior distribution $\rho_{v, \ldots, v}$ be defined as the product of $s$ multivariate Gaussian distributions in $\mathbb{R}^d$ each with mean $v$ and covariance $\beta^{-1}I_d$. For $(\theta_1, \ldots, \theta_s)$ distributed according to $\rho_{v, \ldots, v}$, we have
\begin{equation}
\label{eq:tensorexpectation}
\E_{\rho_{v, \ldots, v}}(\theta_1, \ldots, \theta_s) = (v, \ldots, v),
\end{equation}
In what follows, we use the simplifying notation $\rho = \rho_{v, \ldots, v}$ and the marginals of $\rho$ will be denoted by $\rho_{k}$ for all $k = 1, \ldots, s$. Using the additivity of the Kullback-Leibler divergence for product measures, we get
\[
\mathcal{KL}(\rho_{v, \ldots, v}, \mu) = s\beta/2.
\]
Fix $\lambda > 0$ and let $\beta = \mathbf{r}(\Sigma)$. Our plan will be to bound $\psi\left(\lambda\langle v, X\rangle^s\right)$ by 
\[
\E_{\rho}\log\left(1 + \lambda \prod_{k = 1}^s\langle X, \theta_k\rangle + c_s\lambda^{2} \prod_{k = 1}^s\langle X, \theta_k\rangle^2\right),
\]
where $c_s > 0$ depends only on $s$ and some other terms that do not depend on $v$; then we apply Lemma \ref{lem:pacbayes}. This idea is related to the influence function approach used by Catoni \cite{catoni2012challenging, catoni2016pac} in the context of robust estimation. Using \eqref{eq:tensorexpectation} and the first part of Lemma \ref{lem:almostconvex}, we have
\begin{align*}
\psi\left(\lambda\langle v, X\rangle^s\right)  &= \psi\left(\lambda\E_{\rho} \prod_{k = 1}^s\langle X, \theta_k\rangle\right) 
\\
&\le \E_{\rho}\log\left(1 + \lambda \prod_{k = 1}^s\langle X, \theta_k\rangle + \lambda^{2} \prod_{k = 1}^s\langle X, \theta_k\rangle^2\right) + \min\left\{1, \lambda^2\E_{\rho} \prod_{k = 1}^s\langle X, \theta_k\rangle^2/6\right\}.
\end{align*}
We start with the last summand. Conditionally on $X$, the following holds:
\[
\E_{\rho} \prod_{k = 1}^s\langle X, \theta_k\rangle^2 = \prod_{k = 1}^s\E_{\rho_k}\langle X, \theta_k\rangle^2 = \left(\langle X,v\rangle^2 + \beta^{-1}\|X\|^2\right)^s \le 2^{s - 1}(\langle X,v\rangle^{2s} + \beta^{-s}\|X\|^{2s}).
\]
Therefore, we have
\[
\min\left\{1, \lambda^2\E_{\rho} \prod_{k = 1}^s\langle X, \theta_k\rangle^2/6\right\} \le \min\left\{1, 2^{s - 1}\lambda^2\langle X,v\rangle^{2s}/6\right\} + \min\left\{1, 2^{s - 1}\lambda^2\beta^{-s}\|X\|^{2s}/6\right\}.
\]
We observe that conditionally on $X$, for each $k$, the distribution of $\langle X, \theta_k\rangle$ is Gaussian with mean $\langle X, v\rangle$. Since it is symmetric, we have that $\Pr(\langle X, \theta_k\rangle^2 \ge \langle X, v\rangle^2) \ge \frac{1}{2}$. Thus, with probability at least $\frac{1}{2^{s}}$, it holds that
$
\prod_{k = 1}^s\langle X, \theta_k\rangle^2 \ge \langle X, v\rangle^{2s}.
$ 
This observation implies that 
\[
\min\left\{1, 2^{s - 1}\lambda^2\langle X,v\rangle^{2s}/6\right\} \le  2^s\E_{\rho} \min\left\{1, 2^{s - 1}\lambda^2\prod_{k = 1}^s\langle X, \theta_k\rangle^2/6\right\}.
\]
Choosing $a = 2^s\max\{1, 2^{s - 1}/6\}$ in the second part of Lemma \ref{lem:almostconvex}, we have that there is $c_s > 1$ depending only on $s$ such that
\begin{align*}
&\E_{\rho}\log\left(1 + \lambda \prod_{k = 1}^s\langle X, \theta_k\rangle + \lambda^{2} \prod_{k = 1}^s\langle X, \theta_k\rangle^2\right) + 2^s\E_{\rho} \min\left\{1, 2^{s - 1}\lambda^2\prod_{k = 1}^s\langle X, \theta_k\rangle^2/6\right\}
\\
&\qquad \le \E_{\rho}\log\left(1 + \lambda \prod_{k = 1}^s\langle X, \theta_k\rangle + c_s\lambda^{2} \prod_{k = 1}^s\langle X, \theta_k\rangle^2\right).
\end{align*}

We plug $f(X, \theta_1, \ldots, \theta_s) = \log\left(1 + \lambda \prod_{k = 1}^s\langle X, \theta_k\rangle + c_s\lambda^{2} \prod_{k = 1}^s\langle X, \theta_k\rangle^2\right)$ into Lemma \ref{lem:pacbayes}. Using $\log(1 + y) \le y$ for $y \ge -1$, we have
\begin{align*}
\E_{\rho}\log\E\exp\left(f(X, \theta_1, \ldots, \theta_s)\right) &\le \lambda\E_{\rho}\E\prod_{k = 1}^s\langle X, \theta_k\rangle + c_s\lambda^2\E_{\rho}\E\prod_{k = 1}^s\langle X, \theta_k\rangle^2
\\
&= \lambda\E\langle X, v\rangle^s + c_s\lambda^2\E_{\rho}\E\prod_{k = 1}^s\langle X,\theta_k\rangle^2,
\end{align*}
We need to upper bound the last term. Applying H\"{o}lder's inequality and since $\theta_k$ are independent for $k = 1, \ldots, s$, we have
\begin{equation}
\label{eq:holder}
\E_{\rho}\E\prod_{k = 1}^s\langle X,\theta_k\rangle^2 \le \E_{\rho}\prod_{k = 1}^s\left(\E\langle X, \theta_k\rangle^{2s}\right)^\frac{1}{s} \le \E_{\rho}\prod_{k = 1}^s\left(\eta\sqrt{\theta_k^{\top}\Sigma\theta_k}\right)^2 \le (2\eta^2\|\Sigma\|)^s,
\end{equation}
where we used that for our choice $\beta = \mathbf{r}(\Sigma)$,
\[
\E_{\rho_k}\theta_k^{\top}\Sigma\theta_k \le \|\Sigma\| + \beta^{-1}\tr(\Sigma) = 2\|\Sigma\|.
\]
Thus, Lemma \ref{lem:pacbayes} implies, that with probability at least $1 - \exp(-t)$, for all $v \in S^{d-1}$,
\[
\frac{1}{n\lambda}\sum\limits_{i = 1}^n\E_{\rho}\log\left(1 + \lambda \prod_{k = 1}^s\langle X, \theta_k\rangle + c_s\lambda^{2} \prod_{k = 1}^s\langle X, \theta_k\rangle^2\right) \le \E\langle X, v\rangle^s + c_s\lambda(2\eta^2\|\Sigma\|)^s + \frac{s\mathbf{r}(\Sigma) + 2t}{2n\lambda}.
\] 
By the above computations we have on the same event
\[
\frac{1}{n\lambda}\sum\limits_{i = 1}^n\psi\left(\lambda\langle v, X_i\rangle^s\right) \le \E\langle X, v\rangle^s + c_s\lambda(2\eta^2\|\Sigma\|)^s + \frac{s\mathbf{r}(\Sigma) + 2t}{2n\lambda} + \frac{1}{n\lambda}\sum\limits_{i = 1}^n\min\left\{1, \frac{2^{s - 1}\lambda^2\|X_i\|^{2s}}{6\beta^s}\right\} .
\]
It is left to control the last sum. Denote $Y = \min\left\{1, \frac{2^{s - 1}\lambda^2\|X\|^{2s}}{6\beta^s}\right\}$ and let $Y_1, \ldots, Y_n$ be independent copies of $Y$. Observe that $Y \in [0, 1]$, which implies $\var(Y) \le \E Y$.
By the standard Bernstein inequality \cite[Corollary 2.11]{boucheron2013concentration}, with probability at least $1 - \exp(-t)$,
\[
\frac{1}{n\lambda}\sum\limits_{i = 1}^nY_i \le \frac{1}{\lambda}\left(\E Y + \sqrt{\frac{2\E Yt}{n}}+ \frac{2t}{3n}\right) \le \frac{1}{\lambda}\left(2\E Y + \frac{7t}{6n}\right).
\]
Finally, denoting the standard basis in $\mathbb{R}^d$ by $e_1, \ldots, e_d$ and using the triangle inequality, we have
\[
\left(\E\|X\|^{2s}\right)^{\frac{1}{s}} = \left(\E\left(\sum\nolimits_{j = 1}^d\langle e_j, X\rangle^2\right)^s\right)^{\frac{1}{s}} \le \sum\nolimits_{j = 1}^d\left(\E\langle e_j, X\rangle^{2s}\right)^{\frac{1}{s}} \le \eta^2\tr(\Sigma), 
\]
which implies $\E Y \le 2^{s-1}\eta^{2s}\lambda^2\|\Sigma\|^{s}/6$.
Combining these estimates and taking our choice of $\lambda$ into account, the one-sided bound follows.

To finish the proof, we need to get a two-sided bound. Since the function $\psi$ is symmetric, it is enough to consider $\rho_{v, \ldots, v, -v}$ instead of $\rho_{v, \ldots, v}$ as the posterior distribution, for which the same analysis holds. The claim follows.
\end{proof}

\subsubsection{Bounds on the norms of random vectors}

First, consider the sub-Gaussian case. By \cite[Proposition 2.5.2]{Vershynin2016HDP} the sub-Gaussian vector $X$ satisfies for all $q \ge 1$ and any $y \in \mathbb{R}^d$,
\begin{equation}
\label{eq:lql2}
\left(\E|\langle y, X\rangle|^q\right)^{1/q} \le 3\sqrt{q}\|\langle y, X\rangle\|_{\psi_2} \le 3\kappa\sqrt{qy^{\top}\Sigma y}.
\end{equation}
Choosing $q = 2s$, we obtain that when applying Lemma \ref{lem:truncationlemma} we may choose $\eta = 3\kappa\sqrt{2s}$. 

The norm of the sub-Gaussian vector can easily be analyzed using Lemma \ref{lem:concentrationofthenorm}. Indeed, by the union bound, with probability at least $1 - n\exp(-t)$, it holds that
\begin{equation}
\label{eq:normswhpbound}
\max_{i}\|X_i\| \le \left(36\kappa^2\left(\tr(\Sigma) + 2t\|\Sigma\|\right)\right)^{\frac{1}{2}} \le \kappa\sqrt{108\tr(\Sigma)},
\end{equation}
whenever $t \le \mathbf{r}(\Sigma)$. 

Second, we consider the log-concave case. We use Borell's characaterization of log-concave distributions (see e.g., \cite[Lemma 2.3]{adamczak2010quantitative}): If $X$ is a zero mean random vector in $\mathbb{R}^d$ with a log-concave distribution, then for any $y \in S^{d - 1}$,
\begin{equation}
\label{eq:logconcave}
 \|\langle y, X\rangle\|_{\psi_1} \le \kappa\sqrt{y^{\top}\Sigma y},
\end{equation}
where $\kappa$ is a universal constant. We use the symbol $\kappa$ in the proof but absorb it by the generic constant in our final statements. By \eqref{eq:psione} and \eqref{eq:logconcave} we have for any $y \in S^{d - 1}$ and $q \ge 1$,
\begin{equation}
\label{eq:psioneeqv}
\E|\langle y, X\rangle|^q \le 2q!\|\langle y, X\rangle\|_{\psi_1}^q \le 2q^q\|\langle y, X\rangle\|_{\psi_1}^q \le 2(\kappa q)^q(y^{\top}\Sigma y)^{q/2}.
\end{equation}
Choosing $q = 2s$, we obtain that in Lemma \ref{lem:truncationlemma} we may choose $\eta = 4\kappa s$. The second component is the inequality of Paouris \cite{paouris2006concentration}, written in the following form \cite[Theorem 2]{adamczak2014short}: If  $X$ is a random vector in $\mathbb{R}^d$ with a log-concave distribution, then for any $q \ge 1$,
\begin{equation}
\label{eq:paouris}
\left(\E\|X\|^q\right)^{1/q} \le C\bigl(\E\|X\| + \sup\nolimits_{y \in S^{d - 1}}\left(\E|\langle y, X\rangle|^q\right)^{1/q}\bigr),
\end{equation}
where $C > 0$ is a universal constant. For a zero mean random vector $X$ we have $\E\|X\| \le \sqrt{\tr(\Sigma)}$. By  \eqref{eq:psioneeqv} we have $\E|\langle y, X\rangle|^q \le 2(\kappa q)^q\|\Sigma\|^{q/2}$. Thus, by Markov's inequality together with \eqref{eq:paouris}, with probability at least $1 - \exp(-t)$, it holds that $
\|X\| \le eC\left(\sqrt{\tr(\Sigma)} + 2\kappa t\sqrt{\|\Sigma\|}\right).
$
By the union bound, with probability at least $1 - n\exp(-t)$,
\begin{equation}
\label{eq:whpcontrol}
\max_i\|X_i\| \le eC\left(\sqrt{\tr(\Sigma)} + 2\kappa t\sqrt{\|\Sigma\|}\right).
\end{equation}

\subsubsection{Proof of Theorem \ref{thm:logconcavetheorem}}
We start with an upper tail. Let $\lambda > 0$ be a fixed truncation level. We can write the following decomposition similar to the one used in \cite{adamczak2010quantitative}:

\begin{align}
&\sup\limits_{v \in S^{d - 1}}\left(\frac{1}{n}\sum\limits_{i = 1}^n\langle v, X_i\rangle^2 - \E\langle v, X\rangle^2\right) \nonumber
\\
&\le \sup\limits_{v \in S^{d - 1}}\frac{1}{n}\sum\limits_{i = 1}^n\langle v, X_i\rangle^2\ind{\lambda\langle v, X_i\rangle^2 > 1} +  \sup\limits_{v \in S^{d - 1}}\left(\frac{1}{n\lambda}\sum\limits_{i = 1}^n\psi(\lambda\langle v, X_i\rangle^2) - \E\langle v, X\rangle^2\right). 
\label{eq:decomposition}
\end{align}
The analysis of the second term will follow from Lemma \ref{lem:truncationlemma}. Indeed, for $\eta = 8\kappa$, defining 
\[
\lambda = \frac{1}{\|\Sigma\|}\sqrt{\frac{2\mathbf{r}(\Sigma)}{n(8\kappa)^4}},
\]
which will be our choice throughout the proof, we have, with probability at least $1 - 2\exp(-\mathbf{r}(\Sigma))$,
\[
\sup\limits_{v \in S^{d - 1}}\left(\frac{1}{n\lambda}\sum\limits_{i = 1}^n\psi(\lambda\langle v, X_i\rangle^2) - \E\langle v, X\rangle^2\right) \le c \kappa^2\|\Sigma\|\sqrt{\frac{\mathbf{r}(\Sigma)}{n}},
\]
where $c > 0$ is an absolute constant. Since $\kappa$ is an absolute constant in the log-concave case, we will sometimes absorb it by other absolute constants.

The analysis of large summands can be done via a well-known decoupling-chaining argument. This argument leads to the following result.
\begin{lemma}[Proposition 9.4.2 in \cite{Talagrand2014}]
\label{lem:pickypart}
Assume that $Y_1, \ldots, Y_n$ are independent copies of a random vector $Y$ in $\mathbb{R}^d$ such that for all $x \in S^{d-1}$, it holds that
\[
\|\langle x, Y\rangle\|_{\psi_1} \le 1.
\]
There are absolute constants $c_1, c_2 > 0$ such that the following holds. For any $t > 0$, with probability at least $1 - c_1\exp(-t)$, we have uniformly over $k = 1, \ldots, n$,
\[
\sup\limits_{x \in S^{d - 1}}\sup\limits_{\substack{I \subseteq \{1, \ldots, n\}, \\ |I|\le k}}\left(\sum\limits_{i \in I}\langle x, Y_i \rangle^2\right)^{\frac{1}{2}} \le c_2\left(\max\limits_{i}\|Y_i\| + \sqrt{k}\log\left(\frac{en}{k}\right) + t\right).
\]  
\end{lemma}
Our first observation is that the above bound does not depend on $d$. Moreover, the distribution of $Y$ is not necessarily isotopic. Thus, we can adapt this result to our case. Recall that by \eqref{eq:logconcave} for a zero mean, log-concave vector $X$, we have for any $y \in S^{d - 1}$, $
 \|\langle y, X\rangle\|_{\psi_1} \le \kappa\sqrt{\|\Sigma\|}.
$
Thus, denoting $Y = \frac{1}{\kappa\sqrt{\|\Sigma\|}}X$, we have $\|\langle x, Y\rangle\|_{\psi_1} \le 1$ for all $x \in S^{d - 1}$. By Lemma \ref{lem:pickypart}, with probability at least $1 - c_1\exp(-t)$, it holds that
\begin{equation}
\label{eq:leadingcoordinates}
\sup\limits_{y \in S^{d - 1}}\sup\limits_{\substack{I \subseteq \{1, \ldots, n\}, \\ |I|\le k}}\left(\sum\limits_{i \in I}\langle y, X_i \rangle^2\right)^{\frac{1}{2}} \le c_2\max\limits_{i}\|X_i\|  + c_2\kappa\sqrt{\|\Sigma\|}\left(\sqrt{k}\log\left(\frac{en}{k}\right) + t\right).
\end{equation}
Throughout the proof we choose $t = \left(\mathbf{r}(\Sigma)n\right)^{1/4}$.
From now on we can follow the arguments in \cite{adamczak2010quantitative, Talagrand2014} with several modifications needed to take the effective rank into account. Observe that by \eqref{eq:whpcontrol}, with probability at least 
\[
1 - n\exp(-2\left(\mathbf{r}(\Sigma)n\right)^{1/4}) \ge 1 - \exp(-\left(\mathbf{r}(\Sigma)n\right)^{1/4}) ,
\] 
it holds that
\begin{equation}
\label{eq:normcontrol}
\max_i\|X_i\| \le eC(1 + 4\kappa)\left(\|\Sigma\|^2\mathbf{r}(\Sigma)n\right)^{1/4},
\end{equation}
where  we used that $\sqrt{\tr(\Sigma)} = \sqrt{\|\Sigma\|\mathbf{r}(\Sigma)} \le \left(\|\Sigma\|^2\mathbf{r}(\Sigma)n\right)^{1/4}$ for $n \ge \mathbf{r}(\Sigma)$. Denote
\[
m = \sup\nolimits_{v \in S^{d - 1}}|\{i \in \{1, \ldots, n\}: \lambda\langle v, X_i \rangle^2 > 1\}|.
\] 

\paragraph{Case 1.} On the event where \eqref{eq:leadingcoordinates} and \eqref{eq:normcontrol} hold, we have that $m \le \frac{\mathbf{r}(\Sigma)}{\log^2\left(\frac{e^2n}{\mathbf{r}(\Sigma)}\right)}$. Observe that $m \mapsto \sqrt{m}\log\left(\frac{e^2n}{m}\right)$ is increasing for $1 \le m \le n$ and since $n \ge \mathbf{r}(\Sigma)$,
\[
\sqrt{m}\log\left(\frac{en}{m}\right) \le \sqrt{\mathbf{r}(\Sigma)} + \sqrt{\mathbf{r}(\Sigma)}\frac{\log\log^2\left(\frac{e^2n}{\mathbf{r}(\Sigma)}\right)}{\log\left(\frac{e^2n}{\mathbf{r}(\Sigma)}\right)} \le 2\sqrt{\mathbf{r}(\Sigma)} \le 2\left(n\mathbf{r}(\Sigma)\right)^{1/4}.
\]
Thus, on the same event \eqref{eq:leadingcoordinates} and \eqref{eq:normcontrol} imply
\begin{align*}
\sup\limits_{v \in S^{d - 1}}\sum\limits_{i = 1}^n\langle v, X_i\rangle^2\ind{\lambda\langle v, X_i\rangle^2 > 1} &\le \sup\limits_{v \in S^{d - 1}}\sup\limits_{\substack{I \subseteq \{1, \ldots, n\}, \\ |I|\le m}}\left(\sum\limits_{i \in I}\langle v, X_i \rangle^2\right) \\\
&\le c_2^2\left(\max\limits_{i}\|X_i\|  + 3\kappa\sqrt{\|\Sigma\|}\left((n\mathbf{r}(\Sigma))^{1/4}\right)\right)^{2} 
\\
&\le c_3\|\Sigma\|\sqrt{n \mathbf{r}(\Sigma)},
\end{align*}
where $c_3 > 0$ is an absolute constant. 

\paragraph{Case 2.} Otherwise, we assume $m > \frac{\mathbf{r}(\Sigma)}{\log^2\left(\frac{e^2n}{\mathbf{r}(\Sigma)}\right)}$.
By \eqref{eq:leadingcoordinates}, \eqref{eq:normcontrol} and the union bound, we have, with probability at least $1 - (1 + c_1)\exp(-(\mathbf{r}(\Sigma)n)^{1/4})$,
\begin{equation}
\label{eq:boundonm}
\sqrt{m\lambda^{-1}} \le ec_2C(1 + 2\kappa)\left(\|\Sigma\|^2\mathbf{r}(\Sigma)n\right)^{1/4}  + c_2\kappa\sqrt{\|\Sigma\|}\left(\sqrt{m}\log\left(\frac{en}{m}\right) + (\mathbf{r}(\Sigma)n)^{1/4}\right).
\end{equation}
Assume that the sample size $n$ is large enough, so that $\lambda^{-1} \ge 4\left(c_2\kappa\sqrt{\|\Sigma\|}\log\left(\frac{en}{m}\right)\right)^2$. Solving \eqref{eq:boundonm} we get, in particular, that
\begin{equation}
\label{eq:secondboundonm}
m\log^2\left(\frac{en}{m}\right) \le \frac{c_4\sqrt{\mathbf{r}(\Sigma)n}}{4c_2^2\kappa^2} \quad \text{and} \quad c_2\kappa\sqrt{\|\Sigma\|}\sqrt{m}\log\left(\frac{en}{m}\right) \le c_5\left(\|\Sigma\|^2\mathbf{r}(\Sigma)n\right)^{1/4},
\end{equation}
where $c_4, c_5 > 0$ are absolute constants. Combining \eqref{eq:leadingcoordinates}, \eqref{eq:secondboundonm} and dividing both sides by $n$, we show that for some $c_6 > 0$, on the corresponding event
\[
\sup\limits_{v \in S^{d - 1}}\frac{1}{n}\sum\limits_{i = 1}^n\langle v, X_i\rangle^2\ind{\lambda\langle v, X_i\rangle^2 > 1} \le \sup\limits_{v \in S^{d - 1}}\sup\limits_{\substack{I \subseteq \{1, \ldots, n\}, \\ |I|\le m}}\left(\frac{1}{n}\sum\limits_{i \in I}\langle v, X_i \rangle^2\right) \le c_6\|\Sigma\|\sqrt{\frac{\mathbf{r}(\Sigma)}{n}}.
\]
It is only left to check that $\lambda^{-1} \ge 4\left(c_2\kappa\sqrt{\|\Sigma\|}\log\left(\frac{en}{m}\right)\right)^2$, which is
\begin{equation}
\label{eq:needtocheck}
\sqrt{\frac{n(8\kappa)^4}{2\mathbf{r}(\Sigma)}} \ge 4\left(c_2\kappa\log\left(\frac{en}{m}\right)\right)^2.
\end{equation}
Since we assumed that $m > \frac{\mathbf{r}(\Sigma)}{\log^2\left(\frac{e^2n}{\mathbf{r}(\Sigma)}\right)}$,  one can easily show that if $n \ge c_7\mathbf{r}(\Sigma)$ for some $c_7 > 0$, then the inequality \eqref{eq:needtocheck} holds. Indeed, $x \mapsto \sqrt{x}$ grows faster than $x \mapsto \log^2(x)$.
We proved the upper tail bound.

Although one can prove the lower tail bound completely analogously, we discuss an alternative argument.
For the same value of $\lambda$, using the definition of the truncation function and since $\langle v, X\rangle^2 \ge 0$, we have, with probability at least $1 - 2\exp(-\mathbf{r}(\Sigma))$,
\begin{align*}
\sup\limits_{v \in S^{d - 1}}\left(\E\langle v, X\rangle^2 - \frac{1}{n}\sum\limits_{i = 1}^n\langle v, X_i\rangle^2\right) &\le \sup\limits_{v \in S^{d - 1}}\left(\E\langle v, X\rangle^2 - \frac{1}{n\lambda}\sum\limits_{i = 1}^n\psi(\lambda\langle v, X_i\rangle^2) \right) 
\\&\le c \kappa^2\|\Sigma\|\sqrt{\frac{\mathbf{r}(\Sigma)}{n}}.
\end{align*}
Applying the union bound we finish the proof. We discuss a related argument in Section \ref{sec:lowertail}.
\qed
\subsubsection{Proof of Proposition \ref{cor:logconcavetheorem}}
There are only a few changes compared to the proof of Theorem \ref{thm:logconcavetheorem}. First, an analog of the norm bound \eqref{eq:normcontrol} is implied by our assumption. When applying Lemma \ref{lem:truncationlemma} to the second term in \eqref{eq:decomposition} we choose
\[
\lambda = \frac{1}{\|\Sigma\|}\sqrt{\frac{\mathbf{r}(\Sigma) + t}{n(8\kappa)^4}}.
\]
Our second modification is that we consider the cases: $m \le \frac{\mathbf{r}(\Sigma) + t}{\log^2\left(\frac{e^2n}{\mathbf{r}(\Sigma) + t}\right)}$ and $m > \frac{\mathbf{r}(\Sigma) + t}{\log^2\left(\frac{e^2n}{\mathbf{r}(\Sigma) + t}\right)}$. In the first case $
\sqrt{m}\log\left(\frac{en}{m}\right) \le 2\sqrt{\mathbf{r}(\Sigma) + t}.
$
In this case, with probability at least $1 - c_1\exp(-t)$, it holds that
\[
\sup\limits_{v \in S^{d - 1}}\sum\limits_{i = 1}^n\langle v, X_i\rangle^2\ind{\lambda\langle v, X_i\rangle^2 > 1} \le c_2^2\kappa^2\|\Sigma\|\left((\mathbf{r}(\Sigma)n)^{\frac{1}{4}} + 2\sqrt{\mathbf{r}(\Sigma) + t} + t\right)^2,
\]
Otherwise, if $m > \frac{\mathbf{r}(\Sigma) + t}{\log^2\left(\frac{e^2n}{\mathbf{r}(\Sigma) + t}\right)}$ following the same lines, we show that
\[
\sup\limits_{v \in S^{d - 1}}\frac{1}{n}\sum\limits_{i = 1}^n\langle v, X_i\rangle^2\ind{\lambda\langle v, X_i\rangle^2 > 1} \le c_3\kappa^2\|\Sigma\|\left(\sqrt{\mathbf{r}(\Sigma)n} + t^2\right),
\]
where $c_3 > 0$ is an absolute constant. When $m > \frac{\mathbf{r}(\Sigma) + t}{\log^2\left(\frac{e^2n}{\mathbf{r}(\Sigma) + t}\right)}$, the condition \eqref{eq:needtocheck} will be rewritten as
\[
\sqrt{\frac{n(8\kappa)^4}{\mathbf{r}(\Sigma) + t}} \ge 4\left(c_2\kappa\log\left(\frac{en}{m}\right)\right)^2,
\]
which implies that it is sufficient to take $n \ge c_4(\mathbf{r}(\Sigma) + t)$. Combining these bounds as in the proof of Theorem \ref{thm:logconcavetheorem}, we prove the claim.
\qed
\subsubsection{Proof of Theorem \ref{thm:tensortheorem}}
\paragraph{The log-concave case.}
The proof in the log-concave case is quite similar to the proof of Theorem \ref{thm:logconcavetheorem}. Let us start with an upper tail. Let $\lambda_1 > 0$ be a fixed truncation level. As above, we write the following decomposition:
\begin{align*}
&\sup\limits_{v \in S^{d - 1}}\left(\frac{1}{n}\sum\limits_{i = 1}^n\langle v, X_i\rangle^s - \E\langle v, X\rangle^s\right)
\\
&\le \sup\limits_{v \in S^{d - 1}}\frac{1}{n}\sum\limits_{i = 1}^n|\langle v, X_i\rangle|^s\ind{\lambda_1|\langle v, X_i\rangle|^s > 1} +  \sup\limits_{v \in S^{d - 1}}\left(\frac{1}{n\lambda_1}\sum\limits_{i = 1}^n\psi(\lambda_1\langle v, X_i\rangle^s) - \E\langle v, X\rangle^s\right).
\end{align*}
The analysis of the second term will follow from Lemma \ref{lem:truncationlemma}. For $\eta = 4s\kappa$, defining 
$
\lambda_1 = \sqrt{\frac{2\mathbf{r}(\Sigma)}{n(4s\kappa)^{2s}\|\Sigma\|^s}},
$
we have, with probability at least $1 - 2\exp(-\mathbf{r}(\Sigma))$,
\[
\sup\limits_{v \in S^{d - 1}}\left(\frac{1}{n\lambda_1}\sum\limits_{i = 1}^n\psi(\lambda_1\langle v, X_i\rangle^2) - \E\langle v, X\rangle^2\right) \le C_s(4s\kappa)^s\|\Sigma\|^{s/2}\sqrt{\frac{\mathbf{r}(\Sigma) + t}{n}},
\]
where $C_s$ depends only on $s$.
Observe that $\lambda_1|\langle v, X_i\rangle|^s > 1$ is equivalent to $\langle v, X_i\rangle^2 > \lambda_1^{-2/s}$. Denoting $\lambda^{-1} = \lambda_1^{-2/s}$, we repeat the arguments of the proof of Theorem \ref{thm:logconcavetheorem} for the term
\[
\sup\limits_{v \in S^{d - 1}}\frac{1}{n}\sum\limits_{i = 1}^n\langle v, X_i\rangle^2\ind{\lambda\langle v, X_i\rangle^2 > 1}
\]
with just two modifications: We use instead that, with probability at least $1 - n\exp(-\sqrt{\mathbf{r}(\Sigma)})$,
\[
\max_i\|X_i\| \le eC(1 + 2\kappa)\sqrt{\tr(\Sigma)},
\]
and choose $t = \sqrt{\mathbf{r}(\Sigma)}$ when applying Lemma \ref{lem:pickypart} in the form \eqref{eq:leadingcoordinates}. The same lines imply that if $\lambda_1^{-2/s} \ge 4\left(c_2\kappa\sqrt{\|\Sigma\|}\log\left(\frac{en}{m}\right)\right)^2$ for $m \ge \frac{\mathbf{r}(\Sigma)}{\log^2\left(\frac{e^2n}{\mathbf{r}(\Sigma)}\right)}$, which is
\begin{equation}
\label{eq:secondconditiononn}
\left(\frac{n(4s\kappa)^{2s}\|\Sigma\|^s}{2\mathbf{r}(\Sigma)}\right)^{\frac{1}{s}}\ge 4\left(c_2\kappa\sqrt{\|\Sigma\|}\log\left(\frac{en}{m}\right)\right)^2,
\end{equation}
then, with probability at least $1 - (n + c_1)\exp(-\sqrt{\mathbf{r}(\Sigma)})$,
\begin{align*}
\sup\limits_{v \in S^{d - 1}}\left(\sum\limits_{i = 1}^n|\langle v, X_i\rangle|^s\ind{|\langle v, X_i\rangle|^s > \lambda_1^{-1}}\right)^{\frac{1}{s}} &\le \sup\limits_{v \in S^{d - 1}}\left(\sum\limits_{i = 1}^n\langle v, X_i\rangle^2\ind{\langle v, X_i\rangle^2 > \lambda^{-1}}\right)^{\frac{1}{2}}
\\
&\le c_8\sqrt{\|\Sigma\|\mathbf{r}(\Sigma)},
\end{align*}
where $c_8$ is an absolute constant.
This inequality will contribute the term $\left(c_8\sqrt{\|\Sigma\|\mathbf{r}(\Sigma)}\right)^s$ to the final bound. For this term to be consistent with our final bound we need 
\[
\frac{(c_8\sqrt{\|\Sigma\|\mathbf{r}(\Sigma)})^s}{n} \le \|\Sigma\|^{s/2}\sqrt{\frac{\mathbf{r}(\Sigma)}{n}}.
\]
It is easy to verify that both the last inequality and  \eqref{eq:secondconditiononn} are satisfied when \eqref{eq:assumptiononn} is satisfied. Indeed, in \eqref{eq:secondconditiononn} we are essentially comparing $\left(\frac{n}{\mathbf{r}(\Sigma)}\right)^{\frac{1}{s}}$ with $\log^2\left(\frac{en}{\mathbf{r}(\Sigma)}\right)$, which implies the need for an $s$-dependent factor $c_s$ in \eqref{eq:assumptiononn}.

The analysis of the lower tail can be done analogously. Indeed, we may repeat the analysis for the following decomposition:
\begin{align*}
&\sup\limits_{v \in S^{d - 1}}\left(\E\langle v, X\rangle^s - \frac{1}{n}\sum\limits_{i = 1}^n\langle v, X_i\rangle^s\right)
\\
&\le \sup\limits_{v \in S^{d - 1}}\left(\frac{1}{n}\sum\limits_{i = 1}^n|\langle v, X_i\rangle|^s\ind{\lambda_1|\langle v, X_i\rangle|^s > 1}\right) +  \sup\limits_{v \in S^{d - 1}}\left(\E\langle v, X\rangle^s - \frac{1}{n\lambda_1}\sum\limits_{i = 1}^n\psi(\lambda_1\langle v, X_i\rangle^s)\right).
\end{align*}
The second part of the claim follows. 
\paragraph{The sub-Gaussian case.} We mainly follow the previous proof with several modifications. First, when applying Lemma \ref{lem:truncationlemma} we can use $\eta = 3\kappa\sqrt{2s}$ as discussed above. Second, we use (see \cite[Lemma 2.7.7]{Vershynin2016HDP})
\[
\sqrt{\log 2}\|\langle y, X\rangle\|_{\psi_1} \le \|\langle y, X\rangle\|_{\psi_2} \le \kappa\sqrt{\|\Sigma\|},
\]
and repeat the same lines with \eqref{eq:normcontrol} replaced by \eqref{eq:normswhpbound} to achieve our bound.

Our final observation is that if $n$ is large, then, with high probability,
$
\lambda\max_{i}|\langle v, X_i\rangle|^{s} \le \lambda\max_{i}\|X_i\|^{s} \le 1,
$
where $\lambda$ is chosen as in Lemma \ref{lem:truncationlemma}. On the corresponding event, we have 
\begin{align*}
\sup\limits_{v \in S^{d - 1}}\left|\frac{1}{n}\sum\limits_{i = 1}^n\langle v, X_i\rangle^s - \E\langle v, X\rangle^s\right| &= \sup\limits_{v \in S^{d - 1}}\left|\frac{1}{n\lambda}\sum\limits_{i = 1}^n\psi(\lambda\langle v, X_i\rangle^s) - \E\langle v, X\rangle^s\right|
\\
&\le C_s(3\kappa\sqrt{2s})^s\|\Sigma\|^{s/2}\sqrt{\frac{2\mathbf{r}(\Sigma)}{n}},
\end{align*}
where $C_s$ depends only on $s$.
By \eqref{eq:normswhpbound}, with probability at least $1 - n\exp(-\mathbf{r}(\Sigma))$, we have $\lambda\max_{i}\|X_i\|^{s} \le 1$ when
\[
\sqrt{\frac{2\mathbf{r}(\Sigma)}{n(3\kappa\sqrt{2s})^{2s}\|\Sigma\|^s}} \left(108\kappa^2\tr(\Sigma)\right)^{\frac{s}{2}} \le 1,
\]
Solving this with respect to $n$, we see that it is enough to take $n \ge c_9(\mathbf{r}(\Sigma))^{s + 1}$, where $c_9 > 0$ is an absolute constant.
Applying the union bound, we complete the proof of the statement.
\qed
\begin{remark}
To improve the tail estimate in the Gaussian case for smaller values of $n$, one needs to prove a $\psi_2$-version of Lemma \ref{lem:pickypart}. In our analysis we simply used that $\sqrt{\ln 2}\|\langle y, X\rangle\|_{\psi_1} \le \|\langle y, X\rangle\|_{\psi_2}$.
\end{remark}

\section{Additional results}
\label{sec:additionalresults}
We provide two results extending the ideas used in Theorem \ref{thm:maintheorem} and Theorem \ref{thm:tensortheorem}. 
\subsection{Deviations of the norm of sub-exponential random vectors}
First, assume that $X$ is a zero mean random vector with covariance $\Sigma$ satisfying 
\begin{equation}
\label{eq:subexponential}
\|\langle X, y \rangle\|_{\psi_1} \le \kappa\sqrt{y^{\top}\Sigma y},
\end{equation}
for all $y \in \mathbb{R}^d$. For example, the aforementioned result shows that the assumption \eqref{eq:subexponential} is implied by the log-concavity assumption. We prove the following result, which improves the bound of Alesker \cite{alesker1995psi} on the concentration of the Euclidean norm of a point sampled uniformly at random from a convex body in the isotropic position. Our proof is shorter and does not involve any explicit calculations in polar coordinates. This bound can also be seen as a special instance of the recent result of Lata{\l}a and Nayar \cite{latala2020hadamard} relating the weak and strong moments of random vectors. Although their result is not limited to the sub-exponential assumption \eqref{eq:subexponential}, our bound has the advantage of being dimension-free.

\begin{proposition}
\label{lem:subexponential}
Assume that $X$ is a zero mean sub-exponential random vector \eqref{eq:subexponential}.
Then, for any $t \ge 1$, with probability at least $1 - \exp(-t)$,
\[
\|X\| \le 8\kappa(\sqrt{t\tr(\Sigma)} + t\sqrt{\|\Sigma\|}).
\]
\end{proposition}
\begin{proof}
We follow the proof of Theorem \ref{thm:maintheorem} with a different choice of $\beta$. To control the norm, it is enough to upper bound $\langle X, v\rangle$ uniformly over $v \in S^{d -1}$. We assume without loss of generality that $\Sigma$ is invertible.
Set $\Theta = \mathbb{R}^d$. Let the prior distribution $\mu$ be the multivariate Gaussian distribution with mean zero and covariance $\beta^{-1}\Sigma$. For $u \in \Sigma^{1/2}S^{d - 1}$ let the distribution $\rho_u$ be defined by the density function \eqref{eq:density}. Choosing $r = \sqrt{2\beta^{-1}\tr(\Sigma)}$, similarly to \eqref{kl:boundeq} we have 
\[
\mathcal{KL}(\rho_u, \mu) = \log2 + \beta/2.
\]
For $\lambda > 0$ we want to plug $\lambda \langle X, \Sigma^{-1/2}\theta\rangle$ into Lemma \ref{lem:pacbayes}, where $\theta$ is distributed according to $\rho_u$. Since $\|\theta\| \le \sqrt{\|\Sigma\|} + \sqrt{2\beta^{-1}\tr(\Sigma)}$ almost surely, we have, conditionally on $\theta$, 
\[
\| \langle X,  \Sigma^{-1/2}\theta\rangle\|_{\psi_1} \le \kappa\|\theta\| \le  \kappa(\sqrt{\|\Sigma\|} + \sqrt{2\beta^{-1}\tr(\Sigma)}).
\]
By Lemma \ref{lem:psionenorm}, conditionally on $\theta$,
\[
\log \E_X\exp(\lambda \langle X,  \Sigma^{-1/2}\theta\rangle) \le 4\lambda^2\kappa^2\|\theta\|^2 \le 4\lambda^2\kappa^2(\sqrt{\|\Sigma\|} + \sqrt{2\beta^{-1}\tr(\Sigma)})^2,
\]
whenever $\lambda \le \frac{1}{2\kappa(\sqrt{\|\Sigma\|} + \sqrt{2\beta^{-1}\tr(\Sigma)})}$. We choose $\lambda = \frac{1}{2\kappa(\sqrt{\|\Sigma\|} + \sqrt{2\beta^{-1}\tr(\Sigma)})}$. By Lemma \ref{lem:pacbayes} we have for all $v \in S^{d - 1}$, with probability at least $1 - \exp(-t)$,
\[
\langle X, v\rangle \le 4\lambda\kappa^2(\sqrt{\|\Sigma\|} + \sqrt{2\beta^{-1}\tr(\Sigma)})^2 + \frac{\log 2 + \beta/2 + t}{\lambda}.
\]
Let $\beta = 2t$. On the same event we have, using $t \ge 1$,
\begin{align*}
\|X\| \le (2 + 2\log2)\kappa(\sqrt{\|\Sigma\|} + \sqrt{\tr(\Sigma)}) + 4\kappa(t\sqrt{\|\Sigma\|} + \sqrt{t\tr(\Sigma)}).
\end{align*}
The claim follows.
\end{proof}

The result of Proposition \ref{lem:subexponential} implies the bound of Alesker \cite{alesker1995psi} in the isotopic case (in this case, $\Sigma = I_d$). Importantly, the bound in \cite{alesker1995psi} scales as $\sqrt{td}$ for all values of $t$. An inspection of their proof shows that this happens because for $t \ge d$ one can use the boundedness of a unit volume convex body in the isotropic position. 

Let us shortly explain why the result of Proposition \ref{lem:subexponential} cannot be improved in general. Consider the isotropic case  $\Sigma = I_d$. Using the standard tools, one can rewrite the bound of Proposition \ref{lem:subexponential} as follows: for $p \ge 1$, and some $c > 0$ that depends only on $\kappa$,
\begin{equation}
\label{eq:momentsbound}
\E (\|X\|^p)^{1/p} \le c(\sqrt{pd} + p).
\end{equation}
In \cite{latala2017some} the following example is considered: Let $X = \sqrt{d}ZU$, where $Z$ has the standard Gaussian distribution and $U$ is uniformly distributed on $S^{d - 1}$, independent of $Z$. It is easy to show that \eqref{eq:subexponential} is satisfied for $X$ with $\kappa$ being an absolute constant. In this case $\E(\|X\|^p)^{1/p}$ scales as $\sqrt{pd}$ up to multiplicative constant factors, which proves the necessity of the first term in \eqref{eq:momentsbound}. The necessity of the second term is implied by the case where $d =1$, and $X$ is a standard exponentially distributed random variable.

\subsection{A lower tail version of Theorem \ref{thm:maintheorem}}
\label{sec:lowertail}
It is known that the least singular value of a random matrix can be controlled with high probability guarantees (in fact, it has sub-Gaussian tails) under significantly milder assumptions. This can be seen as a high-dimensional extension of the following fact. For a non-negative random variable $Z$ and for any $t \ge 0$,
\begin{equation}
\label{eq:lowertail}
\Pr(\E Z - Z \ge t) \le \exp(-t^2/(2\E Z^2)).
\end{equation}
That is, the one-sided bound allows for sub-Gaussian concentration even if the random variable $Z$ has infinite moments after the first two. 
There is a list of dimension-dependent lower tail bounds for the least singular value under minimal assumptions \cite{srivastava2013covariance, oliveira2016lower, koltchinskii2015bounding, tikhomirov2016smallest, yaskov2014lower, catoni2016pac, mourtada2019exact}. Our next result improves the dimension-dependent bound of Oliveira \cite{oliveira2016lower} and complements the result of Theorem \ref{thm:maintheorem}. Although one can modify the original proof in \cite{oliveira2016lower}\footnote{This work is first to apply the variational principle to the analysis of the least singular value.} to get a dimension-free bound, we present a short and self-contained proof. 
\begin{proposition}
\label{prop:lowertail}
Assume that $M_1, \ldots M_n$ are independent copies of a positive semi-definite symmetric random matrix $M$ with mean $\E M = \Sigma$. Let $M$ satisfy for some $\kappa \ge 1$,
\begin{equation}
\label{eq:elltwoellone}
\sqrt{\E(x^{\top} M x)^2} \le \kappa^2\ x^{\top}\Sigma x,
\end{equation}
for all $x \in \mathbb{R}^d$.
Then, for any $t > \log 2$, with probability at least $1 - 2\exp(-t)$, it holds that
\[
\sup\limits_{v \in S^{d - 1}}\left(v^{\top}\Sigma v - \frac{1}{n}\sum\limits_{i = 1}^nv^{\top}M_iv\right) \le 7\kappa^2\|\Sigma\|\sqrt{\frac{\mathbf{r}(\Sigma) + t}{n}}.
\]
\end{proposition}
\begin{proof}
Set $\Theta = \mathbb{R}^d$. Let the prior distribution $\mu$ be a multivariate Gaussian distribution with mean zero and covariance $\beta^{-1}I_d$. For $v \in S^{d - 1}$ let $\rho_v$ be a multivariate Gaussian distribution with mean $v$ and covariance $\beta^{-1}I_d$. We have 
$
\mathcal{KL}(\rho_v, \mu) = \beta/2.
$

Let $\theta$ be distributed according to $\rho_v$. Given a positive semi-definite matrix $C$, we have
\begin{equation}
\label{eq:wow}
\E_{\rho_v}\theta^{\top}C\theta = v^{\top}Cv + \beta^{-1}\tr(C).
\end{equation}
Let $Z$ be a random vector having a multivariate Gaussian distribution with zero mean and covariance $\beta^{-1}\Sigma$. Define $\varphi = \Sigma^{1/2}v$. We have, using $(a + b)^4 \le 8a^4 + 8b^4$ and the formulas for the moments of the Gaussian distribution,
\begin{align}
\label{eq:forthmom}
\E_{\rho_v}(\theta^{\top} \Sigma \theta)^2 &= \E\left(\|Z + \varphi\|\right)^4 \le 8\left(3\beta^{-2}\sum\limits_{i = 1}^d \Sigma_{ii}^2+2\beta^{-2}\sum\limits_{i < j}(\Sigma_{ii}\Sigma_{jj} + 2\Sigma^2_{ij}) \right) + 8 \|\varphi\|^4. \nonumber
\\
&=8(\beta^{-2}(2\tr(\Sigma^2) + \tr(\Sigma)^2) + \|\Sigma\|^2) \le 24\beta^{-2}\tr(\Sigma)^2 + 8\|\Sigma\|^2.
\end{align}
Consider the negative part of the truncation function. That is, let $\psi: (-\infty, 0] \to [-1, 0]$ be defined as follows:
\begin{equation}
\psi(x) = 
    \begin{cases}
      x,\quad \textrm{for}\; x \in [-1, 0];
      \\
      -1,\quad \textrm{for}\; x < -1.
    \end{cases}
\end{equation}
One can easily check that for $x \le 0$ it holds that
$
x \le \psi(x) \le \log(1 + x + x^2/2).
$
Fix $\lambda > 0$ and consider the function $\psi(-\lambda\theta^{\top} M \theta)$. Observe that $-\lambda\theta^{\top} M \theta \le 0$ almost surely, so that the function is well defined. We want to plug it into Lemma \ref{lem:pacbayes}. Using $\psi(x) \le \log(1 + x + x^2/2)$ for $x \le 0$ and $\log(1 + y) \le y$ for $y > -1$, we  have
\begin{align*}
\E_{\rho_{v}}\log(\E\exp(\psi(-\lambda \theta^{\top} M \theta))) &\le \E_{\rho_{v}}\log\left(\E(1 -\lambda \theta^{\top} M \theta + \lambda^2 (\theta M \theta)^2/2)\right)
\\
&\le -\lambda\E_{\rho_{v}}\E\theta^{\top} M \theta + \lambda^2\E\E_{\rho_{v}} (\theta^{\top} M \theta)^2/2
\\
&\le -\lambda v^{\top} \Sigma v -  \lambda\beta^{-1}\tr(\Sigma) + \lambda^2\kappa^4\E_{\rho_v}(\theta^{\top} \Sigma \theta)^2/2
\\
&\le -\lambda v^{\top} \Sigma v -  \lambda\beta^{-1}\tr(\Sigma) + \lambda^2\kappa^4(12\beta^{-2}\tr(\Sigma)^2 + 4\|\Sigma\|^2),
\end{align*}
where we used the moment equivalence assumption together with \eqref{eq:wow} and \eqref{eq:forthmom}. By Lemma \ref{lem:pacbayes} we have, with probability at least $1 - \exp(-t)$, simultaneously for all $v \in S^{d - 1}$,
\begin{equation}
\label{eq:pacbayesfull}
\frac{1}{n}\E_{\rho_v}\sum\limits_{i = 1}^n\psi\left(-\lambda \theta^{\top} M_i \theta\right) \le  -\lambda v^{\top} \Sigma v -  \lambda\beta^{-1}\tr(\Sigma)  + \lambda^2\kappa^4(12\beta^{-2}\tr(\Sigma)^2 + 4\|\Sigma\|^2) + \frac{\beta/2 + t}{n}.
\end{equation}
Observe that $\psi$ is convex. Thus, by Jensen's inequality and \eqref{eq:wow} we have
\[
\E_{\rho_v}\psi\left(-\lambda \theta^{\top} M_i \theta\right) \ge \psi\left(-\lambda \E_{\rho_v}\theta^{\top} M_i \theta\right) = \psi(-\lambda v^{\top}M_iv - \lambda \beta^{-1}\tr(M_i))).
\]
Since $v^{\top}M_iv \ge 0$ by the definition of $\psi$ and using $\psi(x) \ge x$, we have
\begin{align*}
\psi(-\lambda v^{\top}M_iv - \lambda \beta^{-1}\tr(M_i))) &= \psi(-\lambda v^{\top}M_iv - \min\{\lambda \beta^{-1}\tr(M_i), 1\})
\\
&\ge -\lambda v^{\top}M_iv - \min\{\lambda \beta^{-1}\tr(M_i), 1\}.
\end{align*}
We plug this in \eqref{eq:pacbayesfull}, divide both sides by $\lambda > 0$ and obtain on the corresponding event 
\begin{align*}
\frac{1}{n}\sum\limits_{i = 1}^n(v^{\top}\Sigma v - v^{\top}M_iv) &\le \frac{1}{n}\sum\limits_{i = 1}^n\left(\min\left\{\lambda^{-1}, \beta^{-1}\tr(M_i)\right\} -  \beta^{-1}\tr(\Sigma)\right) 
\\
&\qquad +  \lambda\kappa^4(12\beta^{-2}\tr(\Sigma)^2 + 4\|\Sigma\|^2)  + \frac{\beta/2 + t}{\lambda n}.
\end{align*}
It is left to analyze the sum involving $\tr(M_i)$. Observe that 
\[
\E \min\left\{\lambda^{-1}, \beta^{-1}\tr(M_i)\right\} \le \beta^{-1}\tr(\Sigma) \quad \textrm{and} \quad \E \left(\min\left\{\lambda^{-1}, \beta^{-1}\tr(M_i)\right\}\right)^2 \le \beta^{-2}\kappa^4(\tr(\Sigma))^2, 
\]
where we used, denoting the standard basis in $\mathbb{R}^d$ by $e_1, \ldots, e_d$, 
\[
\sqrt{\E(\tr(M_i))^2} = \sqrt{\E\left(\sum\nolimits_{j = 1}^d e_j^{\top}M_ie_j \right)^2}  \le \sum\limits_{j = 1}^d\sqrt{\E\left(e_j^{\top}M_ie_j\right)^2} \le \kappa^2\tr(\Sigma).
\]
Applying the Bernstein inequality, we have, with probability at least $1 - \exp(-t)$,
\[
\frac{1}{n}\sum\limits_{i = 1}^n\left(\min\left\{\lambda^{-1}, \beta^{-1}\tr(M_i)\right\} -  \beta^{-1}\tr(\Sigma)\right) \le \beta^{-1}\kappa^2\tr(\Sigma)\sqrt{\frac{2t}{n}} + \frac{2t}{3n\lambda}.
\]
By the union bound, with probability at least $1 - 2\exp(-t)$, for any $v \in S^{d-1}$,
\[
\frac{1}{n}\sum\limits_{i = 1}^n(v^{\top}\Sigma v - v^{\top}M_iv) \le \beta^{-1}\kappa^2\tr(\Sigma)\sqrt{\frac{2t}{n}} +  \lambda\kappa^4(12\beta^{-2}\tr(\Sigma)^2 + 4\|\Sigma\|^2) + \frac{\beta}{2\lambda n} + \frac{5t}{3n\lambda}.
\]
Choose $\beta = \frac{10}{3}\mathbf{r}(\Sigma)$ and $\lambda = \sqrt{\frac{5}{3(12(10/3)^{-2} + 4)}}\frac{1}{\kappa^2\|\Sigma\|}\sqrt{\frac{\mathbf{r}(\Sigma) + t}{n}}$.
Simplifying, we conclude the proof.
\end{proof}
\section{Concluding remarks}
\subsection{A comparison with previous results by Catoni and co-authors}
\label{sec:detailedcomparison}

The application of Lemma \ref{lem:pacbayes} is not new in the context of estimation of quadratic forms. The earliest such application traces back to the works of Audibert and Catoni on robust linear least squares regression \cite{audibert2011robust, audibert2010linear}. More recently, it was done in the context of covariance matrix estimation \cite{catoni2016pac, giulini2018robust, catoni2017dimension} and lower tails for the sample covariance matrix \cite{oliveira2016lower, mourtada2019exact}. With a few exceptions discussed below, this technique has not been previously applied to also control the upper tails of the sample covariance matrix. Applications to general multilinear forms were not previously analyzed.

The closest (in terms of proof techniques) to our results are the bounds appearing in \cite{catoni2016pac} and in the follow-up work \cite{giulini2018robust}. In particular, the work of Catoni \cite{catoni2016pac} provides the dimension-dependent analysis of the sample covariance matrix under two assumptions: the equivalence of $L_4$ and $L_2$ marginal norms as well as the exponential moment assumption on the distribution of $\|X\|$ (see their Proposition 2.2 and Proposition 2.3). However, their results fall short in providing the rate of Proposition \ref{prop:insteadofunionbound} (our simplest bound), as in their case the $\sqrt{d/n}$ convergence rate requires at least $n \ge d^5$ due to the additive term denoted there by $\gamma_{+}$. The sub-optimality comes from the step of the analysis at which $\max\limits_{i = 1, \ldots, n}\|X_i\|$ is used to control the moment generating function of the quadratic form. The problem is explicitly pointed out by Catoni \cite[pages 15 and 16]{catoni2016pac}, where the author is asking if it is possible to show that $\gamma_{+}$ scales as $d/n$ at least in some cases. The same limitations are inherited in the dimension-free extension of this analysis appearing in \cite[Proposition 5.1 and Proposition 5.2]{giulini2018robust}: Apart from the additional \emph{energy} parameter $\sigma$ and logarithmic factors, their bound similarly requires the sample size $n$ to be greater than some integer power of $\mathbf{r}(\Sigma)$. To resolve these problems and achieve the optimal guarantees, we introduce a truncated non-isotropic posterior distribution \eqref{eq:density}. This idea also plays a central role in the proof of Proposition \ref{lem:subexponential} for which no non-isotropic analog is known.

Our second key idea is the analysis of the truncated powers of linear forms in Lemma \ref{lem:truncationlemma}. When $s = 2$, the bound of Lemma \ref{lem:truncationlemma} shares some similarities with the uniform bound of Proposition 4.2 in \cite{catoni2018dimension}. However, we work with an explicit truncation applied to powers of linear forms, which allows to combine the bound of Lemma \ref{lem:truncationlemma} with the classical decoupling-chaining argument analyzed in \cite{adamczak2010quantitative}. This plays a central role in the proofs of Theorem \ref{thm:tensortheorem} and Theorem \ref{thm:logconcavetheorem}. 

\subsection{Possible further extensions}
\label{sec:concludingremarks}

In some applications, one is interested in the following generalization of our problem: In the setup of Theorem \ref{thm:maintheorem} let $T$ be a bounded subset of $\mathbb{R}^d$. Assume that we want to provide a high probability upper bound on 
\[
\sup\limits_{v \in T}\left|\frac{1}{n}\sum\limits_{i = 1}^nv^{\top}M_iv - v^{\top}\Sigma v\right|.
\]
Consider the case where $T$ can be approximated by an ellipsoid; that is, assume that there is a symmetric positive definite matrix $\Gamma$ such that $\frac{1}{2}\Gamma^{1/2} S^{d - 1} \subseteq T \subseteq 2\Gamma^{1/2} S^{d - 1}$. In this case, by considering the random matrix $\Gamma^{1/2}M\Gamma^{1/2}$ we can verify that the norm equivalence assumption \eqref{eq:psionellone} holds and the bound of Theorem \ref{thm:maintheorem} is applicable.

Theorem \ref{thm:tensortheorem} allows various extensions that can be achieved using the same approach. For example, with a slight modification of the proof one can write a similar high probability upper bound on 
\[
\frac{1}{n}\sup\limits_{v_1, \ldots, v_{s} \in S^{d - 1}}\sum\limits_{i = 1}^n\left(\prod_{k = 1}^{s}\langle X^k_i, v_k\rangle - \E\prod_{k = 1}^{s}\langle X^k_i, v_k\rangle\right),
\]
where for any given index $i$ the random vectors $X_i^k$ are either sub-Gaussian or log-concave but not necessarily the same for $k = 1, \ldots , s$. Our results can also be written in the same form as the bounds in \cite{guedon2007lp, adamczak2010quantitative, vershynin2011approximating} allowing isometric approximations at any scale $\varepsilon$. To do so, we need to take a smaller value of $\lambda$ in the proof. The restriction on $n$ will be weakened accordingly. 

Another extension is a multilinear version of the lower tail bound of Proposition \ref{prop:lowertail}. When $s$ is even, we have $\langle X, v\rangle^s \ge 0$, so that we can prove the following: With probability at least $1 - \exp(-t)$, it holds that
\[
\sup\limits_{v \in S^{d - 1}}\left(\E\langle X, v\rangle^s - \frac{1}{n}\sum\limits_{i = 1}^n\langle X_i, v\rangle^s\right) \le c_s\eta^s\|\Sigma\|^{s/2}\sqrt{\frac{\mathbf{r}(\Sigma) + t}{n}},
\]
under the only assumption $\left(\E|\langle y, X\rangle|^{2s}\right)^{1/{2s}} \le \eta\sqrt{y^{\top}\Sigma y}$ for all $y \in \mathbb{R}^d$. The proof follows from Lemma \ref{lem:truncationlemma} by the lower tail argument of Theorem \ref{thm:logconcavetheorem}.
This complements the dimension-dependent lower tail bound of Mendelson \cite[Corollary 1.8]{mendelson2021approximating} valid for general $L_s$ norms with $s > 2$.

In our proofs, we used that the indexing set is essentially the product of either unit spheres or ellipsoids in $\mathbb{R}^d$ and the functions we are considering are linear with respect to the elements of the indexing set. Recently, Koltchinskii  \cite{koltchinskii2020asymptotically}  studied the asymptotic properties of smooth functions of the sample covariance operators. It is likely that if the linearity is replaced by a certain smoothness assumption, our techniques are still applicable. 

\paragraph{Acknowledgments.} The author would like to thank Pedro Abdalla, Rados\l aw Adamczak  and Afonso Bandeira for a valuable feedback and
Jaouad Mourtada for many insightful discussions on the topic. The author is funded in part by ETH Foundations of Data Science (ETH-FDS).

\bibliographystyle{abbrv}  
{\footnotesize
\bibliography{mybib}
}
\end{document}